\documentclass[11pt]{article}
\usepackage{amssymb, amsmath}
\usepackage{amsthm}
\usepackage{graphicx}
\usepackage{hyperref}

\title{A Family of Subgradient-Based Methods for Convex Optimization Problems in a Unifying Framework}
\author{Masaru Ito\footnote{corresponding author} \ ({\it ito1@is.titech.ac.jp})\\ Mituhiro Fukuda ({\it mituhiro@is.titech.ac.jp})\\
Department of Mathematical and Computing Sciences,  Tokyo Institute of Technology\\2-12-1-W8-41 Oh-okayama, Meguro, Tokyo 152-8552 Japan}
\date{Research Report B-477\\Department of Mathematical and Computing Sciences\\Tokyo Institute of Technology\\February 2014, revised December 2015}

\usepackage{multirow}

\def\ie{{\rm\it{i.e.}}}
\def\Real{\mathbb{R}}
\newcommand{\dom}[1]{\textrm{dom}\,#1}
\newcommand{\argmin}{\operatornamewithlimits{argmin}}
\newcommand{\argmax}{\operatornamewithlimits{argmax}}

\def\ds{\displaystyle}

\newcommand{\norm}[1]{\|#1 \|}
\newcommand{\fv}[2]{\left\langle#1,#2\right\rangle}

\theoremstyle{plain}
\newtheorem{theorem}{Theorem}[section]
\newtheorem{lemma}[theorem]{Lemma}
\newtheorem{proposition}[theorem]{Proposition}
\newtheorem{corollary}[theorem]{Corollary}
\newtheorem{method}{Method}
\newtheorem{submethod}{Method}[method]
\newtheorem{property}{Property}

\theoremstyle{definition}

\newtheorem{remark}[theorem]{Remark}

\newtheorem{example}[theorem]{Example}

\begin{document}

\maketitle

%%%
% abstract
%%%

\begin{abstract}
We propose a new family of subgradient- and gradient-based methods which converges with optimal complexity for convex optimization problems whose feasible region is simple enough.
This includes cases where the objective function is non-smooth, smooth, have composite/saddle structure, or are given by an inexact oracle model.
We unified the way of constructing the subproblems which are necessary to be solved at each iteration of these methods.
This permitted us to analyze the convergence of these methods in a unified way compared to previous results which required different approaches for each method/algorithm.
Our contribution rely on two well-known methods in non-smooth convex optimization: the mirror-descent method by Nemirovski-Yudin and the dual-averaging method  by Nesterov.
Therefore, our family of methods includes them and many other methods as particular cases.
For instance, the proposed family of classical gradient methods and its accelerations generalize Devolder {\it et al.}'s, Nesterov's primal/dual gradient methods, and Tseng's accelerated proximal gradient methods.
Also our family of methods can partially become special cases of other universal methods, too.
As an additional contribution, the novel extended mirror-descent method removes the compactness assumption of the feasible region and the fixation of the total number of iterations which is required by the original mirror-descent method in order to attain the optimal complexity.
\\

\noindent
\textbf{Keywords:} non-smooth/smooth convex optimization; structured convex optimization; subgradient/gradient \hspace{-1.5truemm}-based proximal method; mirror-descent method; dual-averaging method; complexity bounds
\\

% 90C25: Mathematical programming > Convex programming
% 68Q25: Theory of computing > Analysis of algorithms and problem complexity
% 49M37: Numerical methods > Methods of nonlinear programming type
\end{abstract}

%\tableofcontents

%%%
% section 1: Introduction
%%%

\section{Introduction}
\label{sec-intro}

%% Background
\subsection{Background on the MDM, the DAM, and related methods}
The gradient-based method proposed by Nesterov in 1983 for smooth convex optimization problems brought a surprising class of `optimal complexity' methods with preeminent performance over the classical gradient methods for the worst case instances \cite{Nes83}.
More precisely, the minimization of a smooth convex function, whose gradient is Lipschitz continuous with constant $L$, by these optimal complexity methods ensures an $\varepsilon$-solution for the objective value within $O(\sqrt{LR^2/\varepsilon})$ iterations%
\footnote{It is important to observe that in all of those methods, the iteration complexity is with respect to the convergence rate of the approximate optimal values and not with respect to the approximate optimal solutions.}%
, while the classical gradient methods require $O(LR^2/\varepsilon)$ iterations; $R$ is the distance between an optimal solution and the initial point.

Since then, the Nesterov's optimal complexity method, as well as further improvements and extensions \cite{AT06,Baes09,Nes04,Nes05s}, applied or extended for solving (non-smooth) convex problems \cite{BT12,LLM11,Nes05e,Nes05s,Nes15,Tseng08,Tseng10} with composite structure \cite{BT09,GL12,GL13,Lan12,Nes05e,Nes05s,Nes13} or with the inexact oracle model \cite{DGNs,DGN14} changed substantially the approach on how to solve large-scale structured convex optimization problems arising from machine learning, compressed  sensing, image processing, statistics, {\it etc.}

Now, if we consider the minimization of general non-smooth convex problems, the situation is apparently different.
The optimal complexity in the non-smooth case is $O(M^2R^2/\varepsilon^2)$ iterations for an $\varepsilon$-solution, where $M$ is a Lipschitz constant for the objective function.
Among several methods or variations, there are two well-known optimal complexity methods: the Mirror-Descent Method (MDM) and the Dual-Averaging Method (DAM).

Since the MDM and the DAM are the main motivations of the present
article, we will focus our subsequent discussion on results related to them.

The MDM originally proposed by Nemirovski and Yudin \cite{Nem79} was later recognized as related to the subgradient algorithm by Beck and Teboulle \cite{BT03}.
The MDM ensures the optimal complexity $O(M^2R^2/\varepsilon^2)$ for a fixed total number of iterations if we choose the (weight) parameters which depend also on $M$ and $R$. If we assume the boundedness of the feasible region, further  variants \cite{NL14,NJLS09} guarantee the complexity\footnote{Knowing $M$ and (an upper bound of) the diameter $D (\geq R)$ of the feasible region can guarantee the complexity $O(M^2D^2/\varepsilon^2)$.} $O(1/\varepsilon^2)$ without these requirements.

The DAM proposed by Nesterov \cite{Nes09} and its modification \cite{NS15} allow us to ensure the complexity $O(1/\varepsilon^2)$ even if the feasible region is unbounded (knowing $R$ can further guarantee the optimal complexity $O(M^2R^2/\varepsilon^2)$).
A key idea to obtain this enhancement in the DAM was the introduction of a sequence which we call {\it scaling parameter} $\beta_k$ in this paper.

As a more recent result, the (approximate) gradient-based methods known as the {\it primal} and {\it dual gradient methods} in \cite{DGN14, Nes13} can be interpreted as particular cases of the MDM and the DAM for the corresponding smooth convex problems, respectively, as it will be clear along the article. However, they only ensure the same complexity $O(LR^2/\varepsilon)$ as the classical gradient methods (for smooth problems) and they require distinct approaches to prove each of their rates of convergence.

Many of gradient-based methods for smooth and structured convex problems were also unified and generalized in some way by Tseng \cite{Tseng08,Tseng10}.
The three algorithms proposed there preserve the $O(\sqrt{LR^2/\varepsilon})$ iteration complexity, but they also require separate analysis for each of them.
Particular cases of the Tseng's optimal methods can be seen as accelerated versions of the MDM and the DAM for smooth and structured problems \cite{DGN14,Nes13}.

It will be important to observe that the difference between the MDM and the DAM or related algorithms lay on the construction of subproblems solved at each iteration as it will be clear in Section~\ref{ssec-opt-algs-nonsm}. As far as we know, there is no results formalizing a combined treatment to them to prove their convergence as we will propose in the present article.

%% Contributions
\subsection{Our contributions}
\label{ssec-contribution}
In this paper, we establish a {\it unifying framework of (sub)gradient-based methods}, namely, Methods~\ref{gen-alg-nonsm},~\ref{gen-alg-clas}, or \ref{gen-alg-fast} with Property~\ref{framework}. As an immediate consequence, we generalized some existing methods, specifically the ones listed above, and we provide a unified convergence analysis for all of them.

We will discuss these ideas in details in the next lines.

All of above existing methods require at each iteration the computation of 
minimizer(s) of one (or two) strongly convex function(s), which we call
{\it auxiliary functions}, over a (simple) closed convex domain.
Property~\ref{framework} (Section~\ref{sec-framework}), which we propose in this paper, reflects the common properties that the auxiliary functions of the MDM and the DAM should satisfy to secure optimal convergence rates.
As a byproduct, we propose two strategies to construct sequentially these auxiliary
functions: the {\it extended Mirror-Descent (MD) model} (\ref{auxfunc-eMD}),
which we believe is completely new in the literature, and  
the {\it Dual-Averaging (DA) model} (\ref{auxfunc-DA}).
In fact, we will show that they can be combined in arbitrary order (Proposition~\ref{auxfunc-admit}) for our final purpose.

Method~\ref{gen-alg-nonsm} (a) and (b) (in Section~\ref{ssec-alg-nonsm}) establish the optimal complexity $O(M^2R^2/\varepsilon^2)$ for non-smooth convex problems, while Methods~\ref{gen-alg-clas} and~\ref{gen-alg-fast} (in Section~\ref{ssec-alg-struc}) establish the complexity $O(LR^2/\varepsilon)$ and the optimal one $O(\sqrt{LR^2/\varepsilon})$, respectively, for structured problems which include smooth, composite structure, saddle structure, or inexact oracle model cases.

An essential idea behind the proofs to show the convergence rates is to check the validity of the inequalities ($R_k$) (\ref{est-rel}) (or ($\hat{R}_k$) (\ref{hat-R-nonsm}) for the non-smooth problems and ($\hat{R}'_k$) (\ref{hat-R-smooth}) for the structured problems) employed in these methods. This approach some how  resemble the estimate sequences \cite{AT06,Baes09,Nes04}, the inequality
($R_k$) \cite{Nes05s}, and the inequality (23) \cite{NS15}.

Since our methods are based on the extended MD and/or the DA updates, they
seems quite restrictive, but as Table~\ref{alg-relation} shows,
many of known optimal methods are particular cases or can be
particularized to coincide with our methods.

A clear advantage of our unifying framework over the exiting ones is that
we can prove all the convergences and their rates in a universal
way without specifying the proofs to a particular method/algorithm.
As far as we know, this is the first time that such general treatment unifying the MDM and the DAM
is proposed.

We remark that our approach should be distinguished from the {\it universal} (sub)gradient-based
methods which can be applied simultaneously to non-smooth or smooth 
problems such
as \cite{GL12,GL13,Lan12} or to structured problems which can admit
inexact oracles, weakly smooth functions, {\it etc.} \cite{DGNs,DGN14,Nes15}.
As noted in Remark~\ref{extended-model}, a more broader approach can be established for our framework if we extend the inexact oracle model as discussed in \cite{Ito15}.

Also, as pointed out before, Tseng unified 
many of these methods in three algorithms, but they require different
treatment for each of them.

As a minor contribution, the generalization of the MDM to the extended MDM (Method~\ref{eMD-alg}) guarantees the complexity $O(1/\varepsilon^2)$ with the advantage of not requiring the values $M$, $R$, and the final number of iterations a priori to determine the (weight) parameters in the method (requiring $R$ further ensures the optimal complexity $O(M^2R^2/\varepsilon^2)$).
This drawback was already partially solved in \cite{GL12,GL13,Lan12,NL14,NJLS09},
but our method has additionally the advantage of not requiring the boundedness of the domain.

The structure of this article is as follows.
First, in Section~\ref{ssec-setting}, we define our problem introducing two classes of convex problems: the non-smooth and the structured problems.
We review some existing methods, in particular
the MDM and the DAM for non-smooth objective functions and Tseng's accelerated gradient methods for smooth ones
in Section~\ref{sec-opt-algs}.
In Section~\ref{sec-framework}, we propose Property~\ref{framework}
which represents a framework of auxiliary functions for the development of our methods, as
well as some supporting lemmas. 
We then propose in Section~\ref{sec-nonsm} the general subgradient-based method and 
prove its convergence rate,
in particular for the extended MDM and the DAM. Subsequently, we propose the classical gradient method and the fast gradient method
for the structured problems in Section~\ref{sec-struc}.

%% Table
\begin{table}[htbp]

\caption{Relation between our family of (sub)gradient-based methods and other known methods. The star (*) corresponds to our result. `Complexity' indicates the number of iterations
to obtain an $\varepsilon$-solution when the objective function has no inexactness for its oracle. \cite{Nes13} is included considering that its Lipschitz constant is known in advance. The third problem class `applicable to both'
indicates that the existing methods can be applicable simultaneously to non-smooth, smooth, and  structured problems.}

\vspace{1truemm}
\scalebox{0.81}{
\begin{tabular}{c|c|l|l}
\hline
problem class &
complexity &
\multicolumn{1}{|c|}{some known methods} &
\multicolumn{1}{|c}{generalized methods}
\\
\hline\hline

% non-smooth

\multirow{6}{*}{non-smooth}
& \multirow{6}{*}{\shortstack{optimal \\ $O\left(\frac{M^2R^2}{\varepsilon^2}\right)$}}
	& \multirow{2}{*}{mirror-descent \cite{BT03,Nem79}}
	& *Method \ref{gen-alg-nonsm} (a) with the model (\ref{auxfunc-eMD})
\\
	&&& \quad $\equiv$ extended mirror-descent: Method \ref{eMD-alg}
\\ \cline{3-4}
	&& dual-averaging \cite{Nes09}
	& *Method \ref{gen-alg-nonsm} (a) with the model (\ref{auxfunc-DA})
\\ \cline{3-4}
	&& double averaging \cite{NS15}
	& *Method \ref{gen-alg-nonsm} (b) with the model (\ref{auxfunc-DA})
\\ \cline{3-4}
	&& sliding averaging \cite{NJLS09}
\\
	&& Nedi\'c-Lee's averaging \cite{NL14}
\\ \hline

% structured/smooth

\multirow{8}{*}{\shortstack{structured/\\ smooth}}
& \multirow{2}{*}{\shortstack{\\[-0.5truemm]classical \\ $O\left(\frac{LR^2}{\varepsilon}\right)$}}
	& primal gradient \cite{DGN14,FM81,Nes13}
	& *Method \ref{gen-alg-clas} with the model (\ref{auxfunc-eMD})
\\[1truemm] \cline{3-4}
	&& dual gradient \cite{DGN14,Nes13}
	& *Method \ref{gen-alg-clas} with the model (\ref{auxfunc-DA})
\\[1truemm] \cline{2-4}

& \multirow{8}{*}{\shortstack{optimal \\ $O\left(\sqrt{\frac{LR^2}{\varepsilon}}\right)$}}
	& estimate sequence method \cite{Baes09,Nes04}
\\ \cline{3-4}
	&& Nesterov's method \cite{Nes05s}
	& Tseng's modified method; see \cite[(35-36)]{Tseng08}
\\ \cline{3-4}
	&& interior gradient method \cite{AT06}
	& \multirow{2}{*}{Tseng's method \cite[Algorithm 1]{Tseng08}}
\\
	&& Lan-Luo-Monteiro's method \cite{LLM11}
	& %\multirow{2}{*}{Tseng's method \cite[Algorithm 1]{Tseng08}} % 

\\ \cline{3-4}
	&& FISTA \cite{BT09}
	& Tseng's first APG \cite{Tseng10}
\\ \cline{3-4}
	&& \multirow{2}{*}{Tseng's second APG \cite{Tseng10}}
	& *Method \ref{MD-fast} $\equiv$ Method \ref{gen-alg-fast} with the model (\ref{auxfunc-eMD})
\\
	&& %\multirow{2}{*}{Tseng's second APG \cite{Tseng10}} % 
	& Tseng's method \cite[Algorithm 1]{Tseng08}
\\ \cline{3-4}
	&& \multirow{2}{*}{Tseng's third APG \cite{Tseng10}}
	& *Method \ref{DA-fast} $\equiv$ Method  \ref{gen-alg-fast} with the model (\ref{auxfunc-DA})
\\
	&& %\multirow{2}{*}{Tseng's third APG \cite{Tseng10}} % 
	& Tseng's method \cite[Algorithm 3]{Tseng08}
\\
\hline

% applicable to both

\multirow{3}{*}{\shortstack{applicable \\ to both}}
& \multirow{3}{*}{optimal}
	& fast gradient method \cite{DGN14}
\\
	&& Ghadimi-Lan's method \cite{GL12,GL13,Lan12}
\\
	&& universal gradient method \cite{Nes15}
\\ \hline
\end{tabular}
\label{alg-relation}

}

\end{table}

%% Problem settings
\subsection{Problem setting and assumptions}
\label{ssec-setting}

In this paper, we consider a finite dimensional real vector space 
$E$ endowed with a norm $\norm{\cdot}$.
The dual space of $E$ is denoted by $E^*$ endowed with the dual norm $\norm{\cdot}_*$ defined by
\begin{equation*}
\|s\|_* = \max_{\|x\| \leq 1} \langle s, x \rangle,\quad s \in E^*
\end{equation*}
where $\fv{s}{x}$ denotes the value of $s \in E^*$ at $x \in E$.

We then define a general convex 
optimization problem as:
\begin{equation}
\label{primal-problem}
	\min_{x \in Q} f(x) 
\end{equation}
where $Q$ is a nonempty closed convex, and possibly unbounded, subset of $E$, 
and $f: E \rightarrow \Real \cup \{+\infty\}$ is a proper lower semicontinuous convex function with $Q \subset \dom{f}:=\{x\in E: f(x)<+\infty\}$.
For each $x \in \dom{f}$, the subdifferential of $f$ at $x$ is denoted by $\partial{f}(x):=\{g \in E^* : f(y) \geq f(x)+\fv{g}{y-x},~\forall y\in E\}$.
We assume throughout this paper that the problem (\ref{primal-problem}) always has an {\it optimal solution $x^* \in Q$}, and the 
structure of $Q$ is simple enough or has some special structure which
permits one to solve a subproblem over it with moderate easiness.
See \cite{Nes05s} for some examples.

We introduce the {\it prox-function} $d(x)$ which is used to define the subproblems in our subgradient-based methods.
Let $d : E \rightarrow \Real \cup \{+ \infty\}$ be a  proper lower semicontinuous convex function which satisfies the following properties:
\begin{itemize}
\item $d(x)$ is a strongly convex function on $Q$ with parameter $\sigma > 0$, \ie,
\begin{equation*}
d(\tau x + (1-\tau)y) \leq \tau d(x) + (1-\tau) d(y) - \frac{1}{2}\sigma\tau(1-\tau)\|x-y\|^2,\quad \forall x, y \in Q,~ \forall \tau \in [0,1].
\end{equation*}
\item $d(x)$ is continuously differentiable on $Q$.
\end{itemize}
We assume that $d(x_0) = \min_{x \in Q}d(x) = 0$ for $x_0:=\argmin_{x \in Q}d(x) \in Q$, which is used for the initial point of our methods.\footnote{We can always assume this requirement for an arbitrary point $x_0 \in Q$ by replacing $d(x)$ by $\xi(x_0,x)$.}

We denote by $\xi(z,x)$ the Bregman distance \cite{Bregman} between $z$ and $x$:
\[
\xi(z,x) := d(x) - d(z) - \langle \nabla{d}(z), x-z \rangle,\quad z,x \in Q.
\]
The Bregman distance satisfies $\xi(z,x) \geq \frac{\sigma}{2}\norm{x-z}^2$ for any $x,z \in Q$ by the strong convexity of $d(x)$.

Finally, we often refer $R$ as $\sqrt{\frac{1}{\sigma}d(x^*)}$, $\sqrt{\frac{1}{\sigma}\xi(x_0,x^*)}$, or their upper bounds which quantifies the distance between the optimal solution $x^*$ and the initial point $x_0$ in view of properties $d(x_0)=0$ and $d(x) \geq \frac{\sigma}{2}\norm{x-x_0}^2$ for every $x \in Q$.

We remark that the problem (\ref{primal-problem}) (as well as the objective function $f(x)$ and the feasible region $Q$) and the prox-function $d(x)$ is fixed throughout the paper.

%% Non-smooth and structured problems

In this paper, we particularize the problem (\ref{primal-problem}) into the following two classes for convenience. Observe that each of problems in these classes is equipped with a proper lower semicontinuous convex function $l_f(y;\cdot):\Real^n \to \Real\cup\{+\infty\}$ satisfying $f(x) \geq l_f(y;x),~\forall x \in Q$, which we call a {\it lower convex approximation of $f(x)$} on $Q$ at $y \in Q$.

\begin{itemize}

% Non-smooth problems
\item {\it The class of non-smooth problems}.
We assume that subgradients of the
objective function $f$, $g(y)\in\partial f(y)$, are computable at any 
point $y \in Q$.
We further assume that the optimal solution of the following (sub)problem is computable for any $s \in E^*$ and $\beta > 0$:
\begin{equation}\label{subprob-nonsm}
\min_{x \in Q}\{ \fv{s}{x} + \beta d(x) \}.
\end{equation}
For the non-smooth problem, we define the linear function $l_f(y;x)$ for each fixed $y \in Q$ by
\begin{equation}\label{l_f-nonsm}
l_f(y;x) := f(y)+\fv{g(y)}{x-y}.
\end{equation}

% Structured problems
\item {\it The class of structured problems}.
We assume that the objective function $f(x)$ of the problem (\ref{primal-problem}) has the following structure:
for any $y \in Q$, there exists a proper lower semicontinuous convex function $l_f(y;\cdot):E\to\Real\cup\{+\infty\}$ satisfying the inequalities
\begin{equation} \label{l_f-struc}
l_f(y;x) \leq f(x) \leq l_f(y;x) + \frac{L(y)}{2}\norm{x-y}^2+\delta(y),\quad \forall x \in Q
\end{equation}
for some $L(y)>0$ and $\delta(y) \geq 0$\footnote{$L(\cdot)$ and $\delta(\cdot)$ can be any positive and nonnegative functions in $y$ on $Q$, respectively. However, there is a restriction to ensure an efficient convergence of the proposed methods as discussed in Section~\ref{ssec-conv-struc}.}.
We also assume that for any $y \in Q,~ s \in E^*$, and $\beta > 0$, we can compute the optimal solution of the (sub)problem
\begin{equation} \label{gen-subprob}
\min_{x \in Q}\{l_f(y;x) + \fv{s}{x} + \beta d(x)\}.
\end{equation}
\end{itemize}

The subproblems (\ref{subprob-nonsm}) and (\ref{gen-subprob}) are subproblems solved at each iteration in the (sub)gradient-based methods. Their difficulties depend on the structure of $Q$, the choice of the prox-function $d(x)$, and the definition of $l_f(y;x)$. See \cite{Nes05s} for some examples of the problem (\ref{subprob-nonsm}) and Example~\ref{ex-struc} below for special cases of (\ref{gen-subprob}).

%% Example
The class of structured problems includes the following important cases given in Example~\ref{ex-struc}. Among them, we are particularly interested on smooth problems (i).

\begin{example}[Structured convex problems]\label{ex-struc}
All the cases excepting the last one were already considered in
the literature.
\begin{enumerate}

% Smooth problems
\item[(i)] {\it Smooth problems.} Suppose that the convex objective function $f(x)$ is continuously differentiable on $Q$ and its gradient $\nabla{f}(x)$ is Lipschitz continuous on $Q$ with a constant $L > 0$:
\[
\norm{\nabla{f}(x) - \nabla{f}(y)}_* \leq L \norm{x -y},\quad \forall x,y \in Q.
\]
Then, defining $l_f(y;x) := f(y) + \fv{\nabla{f}(y)}{x-y}$ yields the condition (\ref{l_f-struc}) with $L(\cdot)\equiv L$ and $\delta(\cdot) \equiv 0$. Then subproblem (\ref{gen-subprob}) is
of the form
\begin{equation}\label{subprob-smooth}
\min_{x \in Q}\{f(y) + \fv{s+\nabla{f}(y)}{x-y} + \beta d(x)\},
\end{equation}
which is equivalent to (\ref{subprob-nonsm}) in this case.

% Composite structure
\item[(ii)] {\it Composite structure}. Let the objective function $f(x)$ has the form
\begin{equation}\label{obj-composite}
f(x) = f_0(x) + \varPsi(x)
\end{equation}
where $f_0(x):E \to \Real\cup\{+\infty\}$ is convex and continuously differentiable on $Q$ with Lipschitz continuous gradient and $\varPsi(x):E\to\Real\cup\{+\infty\}$ is a lower semicontinuous convex function with $Q\subset \dom{\varPsi}$.
This structure has significant applications in machine learning, compressed  sensing, image processing, and statistics \cite{BT09,Tseng10}.

Letting $L>0$ be the Lipschitz constant of $\nabla{f_0}$ on $Q$, we can define $l_f(y;x) := f_0(y)+\fv{\nabla{f_0}(y)}{x-y} + \varPsi(x)$ so that we have (\ref{l_f-struc}) with $L(\cdot)\equiv L$ and $\delta(\cdot) \equiv 0$. The corresponding subproblem has the form
\begin{equation*}
\min_{x \in Q} \{ f_0(y) + \fv{s+\nabla{f_0}(y)}{x-y} + \beta d(x) + \varPsi(x) \}.
\end{equation*}

A generalization of classical methods such as proximal gradient method for this model was proposed by Fukushima and Mine~\cite{FM81} (without assuming convexity for $f_0(x)$). The Nesterov's optimal method (\ref{Nes-opt-alg}) can be also generalized for this case \cite{Nes13}.

Smoothing techniques are also an important approach for this example. Nesterov~\cite{Nes05s} showed a significant improvement on the convergence rate for a particular class and Beck and Teboulle~\cite{BT12} proposed an unifying generalization for it.

% Inexact oracle model
\item[(iii)] {\it Inexact oracle model}. Let us assume that our oracle for $f(x)$ has {\it inexactness}~\cite{DGN14}, that is, we can compute $(\bar{f}(y),\bar{g}(y)) \in \Real\times E^*$ at each $y \in Q$ such that
\begin{equation}\label{i-ora}
0 \leq f(x) - (\bar{f}(y) + \langle {\bar{g}(y)}, {x-y} \rangle) \leq \frac{L_y}{2}\norm{x-y}^2 + \delta_y,\quad \forall x \in Q
\end{equation}
is satisfied for some $L_y > 0$ and $\delta_y \geq 0$.
Then defining $l_f(y;x):=\bar{f}(y) + \langle{\bar{g}(y)},{x-y}\rangle$, $L(y) := L_y$, and $\delta(y) := \delta_y$ we have exactly (\ref{l_f-struc}).

This model was investigated in \cite{DGN14} and the {\it primal, dual,} and {\it fast gradient methods} were proposed.
These methods were also implemented in \cite{Nes15} for a particular class of this model equipped with an iterative scheme to estimate the Lipschitz constants $L_y$ at each iteration. The fast gradient methods can be seen as generalizations of the Nesterov's optimal method (\ref{Nes-opt-alg}) to those cases.

% saddle structure
\item[(iv)] {\it Saddle structure}. Let us consider 
an objective function with the following structure:
\[
f(x) = \sup_{u \in U}\phi(u,x)
\]
where $U$ is a compact convex set of a finite dimensional real vector space $E'$ and $\phi : U \times E \rightarrow \Real\cup\{+\infty\}$ is a concave-convex function satisfying the following conditions.
\begin{itemize}
\item $\phi(\cdot,x)$ is a upper semicontinuous concave function for all $x \in Q$.
\item $\phi(u, \cdot)$ is a lower semicontinuous convex function with $Q \subset \dom{\phi(u, \cdot)}$ for all $u \in U$.
\item For all $u \in U$, $\phi(u, \cdot)$ is continuously differentiable on $Q$ and its gradient is Lipschitz continuous on $Q$, \ie, there exists a constant $L_u \geq 0$ such that
\[
\|\nabla_x \phi(u,x_1) - \nabla_x \phi(u, x_2)\|_* \leq L_{u} \|x_1 - x_2\|, \quad \forall x_1,x_2 \in Q.
\]
\item $L:=\max_{u \in U}L_u$ is finite and positive.
\end{itemize}
Then defining
\begin{equation}
\label{l_f-saddle}
l_f(y;x) := \max_{u \in U}\left\{ \phi(u,y) + \langle \nabla_x{\phi}(u,y), x-y \rangle \right\},
\end{equation}
it satisfies condition (\ref{l_f-struc}) with $L(\cdot) \equiv L$, $\delta(\cdot) \equiv 0$, and
we will have the following subproblem:
\[
\min_{x\in Q}\left\{ \max_{u \in U}\{ \phi(u,y) + \fv{s+\nabla_x \phi(u,y)}{x-y}\} + \beta d(x) \right\}.
\]

This case is a generalization of the structured convex problem discussed in \cite{Nes83}, namely, $E'\equiv\Real^m$ and, for each $u=(u^{(1)},\ldots, u^{(m)}) \in U$,  defining $\phi(u,x)=\sum_{i=0}^m u^{(i)}f_i(x)$ for given differentiable convex functions $f_1(x),\ldots,f_m(x)$ on $E$ with Lipschitz continuous gradient. The convexity of $\phi(u,\cdot)$ is satisfied by imposing the following assumption as in \cite{Nes83}: if there exists $u \in U$ such that $u^{(i)} < 0$, then $f_i(x)$ is a linear function. Letting $L^{(i)}$ be a Lipchitz constant of $\nabla{f}_i(x)$ for $i=1,\ldots,m$, we have $L=\max_{u \in U}L_u = \max_{u \in U}\sum_{i=1}^m u^{(i)}L^{(i)}$.

The definition of $l_f(y;x)$ can be simplified when $Q \subset {\rm int}(\dom{f}$) and $\phi(\cdot,x)$ is strictly concave for all $x \in Q$. In this case, denoting $u_x = \argmax_{u \in U}\phi(u,x)$, we have $\nabla{f}(x) = \nabla_x{\phi}(u_x,x)$ and therefore we can define
\[ l_f(y;x):= \phi(u_y,y) + \fv{\nabla_x\phi(u_y,y)}{x-y} \]
which satisfies (\ref{l_f-struc}) with $L(\cdot) \equiv L$ and $\delta(\cdot)\equiv 0$.
Its subproblem is of the form (\ref{subprob-smooth}).
This situation is also discussed in Tseng's methods \cite{Tseng08}.

% Mixed structure
\item[(v)] {\it Mixed structure}. The above examples can be combined with each other; for instance, considering the function $f_0(x)$ in (ii) with inexactness (iii) or with the saddle structure (iv), or considering the function $\phi(u,x)$ in (iv) with inexactness (iii) or with the composite structure (ii) satisfies our requirement (\ref{l_f-struc}).
\end{enumerate}
\end{example}

\begin{remark}
\label{extended-model}
Defining the class of the structured problems with the inexact oracle model (iii) allow us to include non-smooth and weakly smooth convex problems (see \cite{DGNs,DGN14}).
Moreover, considering a generalization of (\ref{l_f-struc}) by replacing $\delta(y)$ with $\delta(x,y)$ where $\delta(\cdot,y)$ is a nonnegative and lower semicontinuous convex function on $Q$ for every $y \in Q$, we can further include other structured convex problems such as the composite convex problem discussed in \cite{GL12,GL13,Lan12} (in the deterministic version). Then the objective function $f(x)$ will satisfy  the condition
\[ f(y)-f(x)-\fv{g(y)}{y-x} \leq \frac{L}{2}\norm{y-x}^2+M\norm{y-x},\quad \forall x,y \in Q, \]
for a subgradient mapping $g(x) \in \partial{f}(x)$, $L,M \geq 0$, and $\delta(x,y):=M\|x-y\|$. Observe that the smooth and the non-smooth problems are its special cases when $M=0$ and $L=0$, respectively. \cite{Ito15} investigates in detail the extensions of the methods discussed here to this setting.
\end{remark}

%%%
% section 2: Existing optimal methods
%%%

\section{Existing optimal methods}
\label{sec-opt-algs}

In this section, we review some well-known subgradient-based and
gradient-based methods. 
In particular, we focus on the Mirror-Descent Method (MDM),
the Dual-Averaging Method (DAM), the double and triple averaging methods
for non-smooth problems in Section~\ref{ssec-opt-algs-nonsm},
and on the Nesterov's accelerated gradient and the Tseng's 
Accelerated Proximal Gradient (APG) methods for 
smooth problems (or the structured ones) in Section~\ref{ssec-opt-algs-sm}.

The purpose of this section is to unify the notation of these existing
methods in order to introduce a unifying framework for them
in Section~\ref{sec-framework}.
For that, we sometimes change
the variables' names, shift
their indices, and add constants in the objective functions
of optimization subproblems when compared to the
original articles.

%%% optimal methods for non-smooth case

\subsection{Optimal methods for non-smooth problems}
\label{ssec-opt-algs-nonsm}
Let us first see some existing methods for non-smooth convex problems. Recall that, from the definition of the class of non-smooth problems, we have, for $y \in Q$, a subgradient mapping $g(y) \in \partial{f}(y)$ and a lower convex approximation $l_f(y;x):=f(y)+\fv{g(y)}{x-y}$ of $f(x)$.

%% The MDM
\subsubsection{The mirror-descent method}
The Mirror-Descent Method (MDM) \cite{Nem79} in the form reinterpreted by Beck and Teboulle \cite{BT03} generates $\{x_k\}_{k \geq 0} \subset Q$ by setting $x_0:=\argmin_{x\in Q}d(x)$ and
\begin{equation}
\label{MD-org-iter}
\begin{array}{rcl}
g_k &:=& g(x_k) \in \partial{f}(x_k),\\
z_k &:=& \argmin_{x \in Q}\{\lambda_k[f(x_k)+\fv{g_k}{x-x_k}]+\xi(x_k,x)\},\\
x_{k+1} &:=& z_k
\end{array}
\end{equation}
for each $k \geq 0$, where $\lambda_k > 0$ is a {\it weight parameter}.

The variable $z_k$ is redundant here, but we keep it in order to use the same notation of our unifying framework.
Notice that, by the definition of the Bregman distance, the computation of $z_k$ reduces to the form of (\ref{subprob-nonsm}).

It is known that the MDM reduces to the classical subgradient method $x_{k+1}:=\pi_Q(x_k-\lambda_k g_k)$ when $E$ is a Euclidean space, $\norm{\cdot}$ is the norm of $E$ induced by its inner product, $d(x):=\frac{1}{2}\norm{x-x_0}^2$, and $\pi_Q$ is the orthogonal projection onto $Q$
(see also Auslender-Teboulle \cite{AT06} and Fukushima-Mine \cite{FM81} for some related works).

The MDM produces the following estimate \cite{BT03}:
\begin{equation}\label{MD-org-est}
\forall k \geq 0,\quad  \Delta_k := \frac{\sum_{i=0}^k \lambda_if(x_i)}{\sum_{i=0}^k\lambda_i} - f(x^*) \leq \frac{\xi(x_0,x^*) + \frac{1}{2\sigma}\sum_{i=0}^k\lambda_i^2\norm{g_i}_*^2}{\sum_{i=0}^k \lambda_i}.
\end{equation}
Therefore, if we define the approximate solution
\begin{equation*}%\label{hat-x-def}
\hat{x}_k:=\frac{\sum_{i=0}^k\lambda_ix_i}{\sum_{i=0}^k\lambda_i},
\end{equation*}
we have by the convexity of $f$ that $f(\hat{x}_k)-f(x^*)\leq \Delta_k$; we can also obtain the estimate $\min_{0 \leq i \leq k}f(x_i)-f(x^*) \leq \Delta_k$.

Furthermore, the right hand side of (\ref{MD-org-est}) can be bounded by $M\sqrt{2\sigma^{-1}\xi(x_0,x^*)}/\sqrt{k+1}$ if $M:=\sup\{\norm{g}_*:g \in \partial{f}(x),~x \in Q\}$ is finite and if we choose the constant weight parameters
\begin{equation}\label{MD-org-weights}
\lambda_i := M^{-1}\sqrt{2\sigma \xi(x_0,x^*)}/\sqrt{k+1},\quad i=0,\ldots,k
\end{equation}
for a fixed $k \geq 0$.
If we further know an upper bound $R \geq \sqrt{\frac{1}{\sigma}\xi(x_0,x^*)}$, this result ensures an $\varepsilon$-solution in $O(M^2R^2/\varepsilon^2)$ iterations which provides the optimal complexity for the non-smooth case \cite{BT03}.
The above choice of weight parameters, however, is impractical since it depends on the final iterate $k$ and an upper bound for $\xi(x_0,x^*)$; a more practical choice $\lambda_i := r/\sqrt{i+1}$ for some $r > 0$ only ensures an upper bound $\frac{\xi(x_0,x^*)+(2\sigma)^{-1}r^2M^2(1+\log(k+1))}{2r(\sqrt{k+2}-1)} = O(\log k/\sqrt{k})$ for the right hand side of (\ref{MD-org-est}). Note that, however, when the feasible region $Q$ is compact, the weight parameters $\lambda_i:=r/\sqrt{i+1}~(r>0)$ ensure the rate $O(1/\sqrt{k})$ of convergence for the difference $f(\hat{x}_k)-f(x^*)$ by considering $\hat{x}_k$, a weighted average of $x_0,\ldots,x_k$ \cite{NL14,NJLS09}.

%% The DAM
\subsubsection{The dual-averaging method and its variants}
\label{sssec-DAM}
The Dual-Averaging Method (DAM) proposed by Nesterov \cite{Nes09} overcomes 
the dependence of weight parameters of the MDM on $k$
and even achieves the rate $O(1/\sqrt{k})$ of convergence.
This method employs non-decreasing positive {\it scaling parameters} $\{\beta_k\}_{k \geq -1}~(\beta_{k+1} \geq \beta_k > 0)$ in addition to the weight parameters $\{\lambda_k\}_{k \geq 0}$.

From the initial point $x_0 := \argmin_{x \in Q}d(x)\in Q$, the DAM is performed by the iteration

\begin{equation}
\label{DA-org-iter}
\begin{array}{rcl}
g_k &:=& g(x_k) \in \partial{f}(x_k),\\
z_k &:=& \argmin_{x \in Q}\left\{\sum_{i=0}^k\lambda_i[f(x_i)+\fv{g_i}{x-x_i}]+\beta_kd(x)\right\},\\
x_{k+1} &:=& z_k
\end{array}
\end{equation}
for each $k \geq 0$.

It is important to note that the difference between the MDM and the DAM is in the construction of the subproblems. Both methods solve subproblems of the form $z_k:=\argmin_{x \in Q}\psi_k(x)$ defined by the {\it auxiliary functions}
\begin{equation}
\label{MD-org-auxfunc}
\psi_k(x) := \lambda_kl_f(x_k;x)+\xi(x_k,x) = \lambda_k[f(x_k)+\fv{g_k}{x-x_k}] + \xi(x_k,x)
\end{equation}
in the MDM and
\begin{equation}
\label{DA-org-auxfunc}
\psi_k(x) := \sum_{i=0}^k \lambda_il_f(x_i;x)+\beta_kd(x) = \sum_{i=0}^k \lambda_i[f(x_i)+\fv{g_i}{x-x_i}] + \beta_kd(x)
\end{equation}
in the DAM.

% estimate of DAM
Nesterov proved that the DAM satisfies the following general estimate (set $D=d(x^*)$ in \cite[Theorem 1 and (3.2)]{Nes09}):
\begin{equation}
\label{gen-bound-DAM}
\forall k \geq 0,\quad \Delta_k := \frac{\sum_{i=0}^k \lambda_if(x_i)}{\sum_{i=0}^k\lambda_i} - f(x^*) \leq \frac{\beta_kd(x^*)+\frac{1}{2\sigma}\sum_{i=0}^k\frac{\lambda_i^2}{\beta_{i-1}}\norm{g_i}_*^2}{\sum_{i=0}^k\lambda_i}.
\end{equation}
In order to ensure the rate $O(1/\sqrt{k})$ of convergence, we do not even need a prior knowledge of an upper bound for $\xi(x_0,x^*)$ in contrast to the MDM;
for instance, choosing $\lambda_k := 1$ and $\beta_k :=\gamma \hat{\beta}_k$ where $\gamma > 0$ and 
\begin{equation}\label{eq:beta}
\hat{\beta}_{-1}:=\hat{\beta}_0:=1,~\hat{\beta}_{k+1}:=\hat{\beta}_k+\hat{\beta}_k^{-1}, \qquad \forall k \geq 0, 
\end{equation}
the estimate (\ref{gen-bound-DAM}) yields
\[ \forall k \geq 0,\quad \Delta_k \leq \left(\gamma d(x^*)+\frac{M^2}{2\sigma \gamma}\right)\frac{0.5+\sqrt{2k+1}}{k+1}. \]
Furthermore, if we know $M$ and $R \geq \sqrt{\frac{1}{\sigma}d(x^*)}$, the choice $\gamma:=\frac{M}{\sqrt{2}\sigma R}$ achieves the optimal iteration complexity $O(M^2R^2/\varepsilon^2)$ to obtain an $\varepsilon$-solution.

A key in the analysis of the DAM in \cite{Nes09} is the use of a dual approach such as the conjugate function of $\beta d(x)$ for $\beta > 0$.
In this paper, we prove the same result with simpler arguments (in Section~\ref{sec-nonsm}) for the  DAM and for (an extension of) the MDM without employing duality.

Nesterov and Shikhman \cite{NS15} further proposed variants of the DAM, the double and triple averaging methods, in order to obtain convergence results for the sequence $\{x_k\}$. The double averaging method \cite[eq. (28)]{NS15} iterates starting from $x_0:=\argmin_{x \in Q}d(x) \in Q$ as follows:
\begin{equation}\label{double-ave}
z_k := \argmin_{x \in Q}\psi_k(x),\quad x_{k+1}:=(1-\tau_k)x_k+\tau_kz_k,\quad k=0,1,2,\ldots
\end{equation}
where $\tau_k:=\lambda_{k+1}/\sum_{i=0}^{k+1}\lambda_i$ and $\psi_k(x)$ is defined by the auxiliary function (\ref{DA-org-auxfunc}) used in the DAM. This method bounds the difference $f(x_k)-f(x^*)$ by the same value as the right hand side of (\ref{gen-bound-DAM}) \cite[Theorem 3.1]{NS15} for all $k \geq 0$. Hence, it achieves optimality. The triple averaging, which is a modification  of (\ref{double-ave}), allows further flexibility on the choices for $\{\lambda_k\}$ and $\{\beta_k\}$ \cite[Theorem 3.3]{NS15}.

Observe that for all the above methods, we do not
need to evaluate any function value at any iteration and 
$x_{k+1}$ is determined uniquely even if $Q$ is unbounded,
since $d(x)$ is strongly convex \cite[Lemma 6]{Nes09}.

%%% optimal methods for smooth case

\subsection{Optimal methods for smooth problems}
\label{ssec-opt-algs-sm}

We now review some existing methods for the smooth problems which are a special case of the structured problems (see Example~\ref{ex-struc} (i)).

Suppose that the function $f(x)$ in (\ref{primal-problem}) is convex, continuously differentiable, and its gradient is Lipschitz continuous on $Q$ with constant $L>0$.

Many optimal complexity methods were proposed in the literature under this assumptions (see, {\it e.g.},  Table~\ref{alg-relation}). In particular, we recall the optimal methods proposed by Nesterov \cite{Nes05s} and Tseng \cite{Tseng08,Tseng10} for a comparison with our results.

Given positive weight parameters $\{\lambda_k\}_{k \geq 0}$, both methods solve either or both of the following subproblems:
\begin{equation*}%\label{two-subprob}
\begin{array}{cl}
{\rm (a)} & \min_{x \in Q}\big\{ \lambda_k[f(x_k)+\fv{\nabla{f}(x_k)}{x-x_k}] + \frac{L}{\sigma}\xi(z_{k-1},x) \big\},\\[1truemm]
{\rm (b)} & \min_{x\in Q}\left\{ \sum_{i=0}^k \lambda_i[f(x_i) + \fv{\nabla{f}(x_i)}{x-x_i}] + \frac{L}{\sigma}d(x) \right\}
\end{array}
\end{equation*}
where $\{x_k\}_{k \geq 0} \subset Q$ is the sequence generated by those methods and $\{z_k\}_{k \geq -1}$ is the sequence of optimal solutions of the subproblem (a) or (b) as specified by the method.
Similarly to the non-smooth case, it is not necessary
to evaluate the function values at $x_k$'s and the minimums of (a) and (b) are uniquely defined.

%% Nesterov's method
The Nesterov's optimal method (see modified method in \cite[Section 5.3]{Nes05s}) with a particular choice for the weight parameters $\lambda_k$ is described as follow. \\[1truemm]
{\bf Nesterov's method}: Set $\lambda_k := (k+1)/2$ for $k\geq 0$ and $x_0 := z_{-1} := \argmin_{x \in Q}d(x)$. Compute the solution $\hat{z}_0$ of (a) with $k=0$ and set $\hat{x}_0:=z_0:=\hat{z}_0$. For $k \geq 0$, iterate the following procedure:
\begin{equation}
\label{Nes-opt-alg}
\begin{array}{ll}
{\rm Set} & x_{k+1}:=(1-\tau_k)\hat{x}_k + \tau_k z_k,\quad {\rm where~~}\tau_k:=\frac{\lambda_{k+1}}{\sum_{i=0}^{k+1}\lambda_i},\\
{\rm Compute}& \hat{z}_{k+1}:=\argmin_{x \in Q}\Big\{\lambda_{k+1}[f(x_{k+1})+\fv{\nabla{f}(x_{k+1})}{x-x_{k+1}}] + \frac{L}{\sigma}\xi(z_{k},x)\Big\},\\
{\rm Set} & \hat{x}_{k+1}:=(1-\tau_k)\hat{x}_k + \tau_k \hat{z}_{k+1},\\
{\rm Compute}& z_{k+1} := \argmin_{x \in Q}\left\{\sum_{i=0}^{k+1} \lambda_i[f(x_i) + \fv{\nabla{f}(x_i)}{x-x_i}] + \frac{L}{\sigma}d(x)\right\}.
\end{array}
\end{equation}

%% Tseng's APG
In comparison, the Tseng's second and third Accelerated Proximal Gradient (APG) methods \cite{Tseng10}, which are particular cases of algorithms 1 and 3 in \cite{Tseng08}, 
only require the computation of either $\hat{z}_k$ or $z_k$ of the Nesterov's method, respectively.\\
{\bf Tseng's second APG method:}
Set $\lambda_0 := 1,~ \lambda_{k+1}:=\frac{1+\sqrt{1+4\lambda_k^2}}{2}$ for $k\geq 0$, and $x_0 := z_{-1} := \argmin_{x \in Q}d(x)$. Compute the solution $z_0$ of (a) with $k=0$ and set $\hat{x}_0:=z_0$. For $k \geq 0$, iterate the following procedure:
\begin{equation}
\label{Tseng-MD}
\begin{array}{ll}
{\rm Set} & x_{k+1}:=(1-\tau_k)\hat{x}_k + \tau_k z_k,\quad {\rm where~~}\tau_k:=\frac{\lambda_{k+1}}{\sum_{i=0}^{k+1}\lambda_i},\\
{\rm Compute}& z_{k+1} := \argmin_{x \in Q}\Big\{\lambda_{k+1}[f(x_{k+1})+\fv{\nabla{f}(x_{k+1})}{x-x_{k+1}}] + \frac{L}{\sigma}\xi(z_{k},x)\Big\},\\
{\rm Set} & \hat{x}_{k+1}:=(1-\tau_k)\hat{x}_k + \tau_k z_{k+1}.
\end{array}
\end{equation}
{\bf Tseng's third APG method:}
Set $\lambda_0 := 1,~ \lambda_{k+1}:=\frac{1+\sqrt{1+4\lambda_k^2}}{2}$ for $k\geq 0$, and $x_0 := z_{-1} := \argmin_{x \in Q}d(x)$. Compute the solution $z_0$ of (b) with $k=0$ and set $\hat{x}_0:=z_0$. For $k \geq 0$, iterate the following procedure:
\begin{equation}
\label{Tseng-DA}
\begin{array}{ll}
{\rm Set} & x_{k+1}:=(1-\tau_k)\hat{x}_k + \tau_k z_k,\quad {\rm where~~}\tau_k:=\frac{\lambda_{k+1}}{\sum_{i=0}^{k+1}\lambda_i},\\
{\rm Compute}& z_{k+1} := \argmin_{x \in Q}\left\{\sum_{i=0}^{k+1} \lambda_i[f(x_i) + \fv{\nabla{f}(x_i)}{x-x_i}] + \frac{L}{\sigma}d(x)\right\},\\
{\rm Set} & \hat{x}_{k+1}:=(1-\tau_k)\hat{x}_k + \tau_k z_{k+1}.
\end{array}
\end{equation}

%% Remark
\begin{remark}
To see the equivalence to the Tseng's second APG method, notice that $x_0$ is
not used at all in \cite{Tseng10}. Then defining $d(x):=D(x,z_0)=\eta(x)-\eta(z_0)-\langle\nabla\eta(z_0),x-z_0\rangle$ for an arbitrary $z_0\in Q$, we have $\sigma=1$ in (a). Finally, making the correspondence $z_k\rightarrow z_{k-1}$, $y_k\rightarrow x_k$, $x_k\rightarrow \hat{x}_k$, and $\theta_k\rightarrow \frac{1}{\lambda_k}$, it will result in our notation. For the Tseng's third APG method, identical observations are valid, excepting that we define $d(x):=\eta(x)-\eta(z_0)$ instead.
\end{remark}

In order to see a connection to the unifying framework of this paper, 
let us focus on the subproblems of the Nesterov's and the Tseng's methods.
These subproblems have the form $z_k:=\argmin_{x \in Q}\psi_k(x)$ with the auxiliary functions
\begin{equation}\label{MD-auxfunc-sm}
\psi_k(x) := \lambda_k l_f(x_k;x) + \frac{L}{\sigma}\xi(z_{k-1},x)
\end{equation}
for the Tseng's second APG method and
\begin{equation}\label{DA-auxfunc-sm}
 \psi_k(x) := \sum_{i=0}^k \lambda_i l_f(x_i;x) + \frac{L}{\sigma}d(x)
\end{equation}
for the Tseng's third APG method, where $l_f(y;x):=f(y)+\fv{\nabla{f}(y)}{x-y}$ (recall Example~\ref{ex-struc} (i)); the Nesterov's method can be seen as their hybrid.
Note that the auxiliary functions (\ref{MD-auxfunc-sm}) and (\ref{DA-auxfunc-sm}) correspond to the one of the MDM (\ref{MD-org-auxfunc}), excepting the factor $L/\sigma$, and the one of the DAM (\ref{DA-org-auxfunc}) with $\beta_k=L/\sigma$, respectively.

It can be shown that both
Nesterov's and Tseng's methods attain the optimal convergence rate;
the Nesterov's method (\ref{Nes-opt-alg}) and the Tseng's third APG method (\ref{Tseng-DA}) satisfy
\begin{equation*}
\forall k \geq 0,~ f(\hat{x}_k) - f(x^*) \leq \frac{4Ld(x^*)}{\sigma (k+1)(k+2)}
\end{equation*}
while the Tseng's second APG method (\ref{Tseng-MD}) satisfies
\begin{equation*}
\forall k \geq 0,~ f(\hat{x}_k) - f(x^*) \leq \frac{4L\xi(x_0,x^*)}{\sigma (k+2)^2}.
\end{equation*}
Note that these convergence rates ensures an $\varepsilon$-solution with the optimal iteration complexity $O(\sqrt{LR^2/\varepsilon})$ where $R=\sqrt{\frac{1}{\sigma}d(x^*)}$ for the first estimate and $R=\sqrt{\frac{1}{\sigma}\xi(x_0,x^*)}$ for the second one, respectively.

The convergence analysis of these three methods are performed in distinct ways.
What we propose in Section~\ref{sec-struc} is a universal analysis for them using the unifying framework defined in Section~\ref{sec-framework}.

The above gradient-based methods for smooth problems can be generalized to a wider class of convex problems.
The Nesterov's method (\ref{Nes-opt-alg}) was generalized for the composite structure \cite{Nes13} and for the inexact oracle model \cite{DGN14}.
The Tseng's methods were originally proposed for the composite objective function unifying some existing methods \cite{AT06,BT09,Nes05s}, while we only have described the particular ones for the smooth case.

It is important to note that the inexact oracle model \cite{DGN14} is also applicable to non-smooth problems yielding optimal subgradient methods; more precisely, it is applicable to `weakly smooth' convex problems (see Remark~\ref{extended-model}).

There are several {\it universal} (sub)gradient methods \cite{DGNs,DGN14,GL12,GL13,Lan12,Nes15} which are optimal for both non-smooth and smooth problems (and further generalized ones). In contrast to such universal methods, we will propose different (not universal) (sub)gradient-based methods for non-smooth and smooth problems which also include some of previously mentioned methods. A key contribution of our approach is that it provides a unified methodology on the analysis of optimal subgradient/gradient-based methods for non-smooth/smooth problems.

%%%
% section 3: a unified framework
%%%

\section{Construction of auxiliary functions in the unifying framework}
\label{sec-framework}

For all methods we reviewed for the non-smooth or the smooth
problems, we need to form one or two {\it 
auxiliary functions} $\psi_k(x)$ and solve the corresponding subproblem(s)
$\min_{x\in Q}\psi_k(x)$ at each iteration.
In this section, we will propose general conditions which these
auxiliary functions should satisfy (Property~\ref{framework}) in order to provide a unifying analysis for them.
In particular, we will see that these auxiliary functions can be
derived from the extended MD model (\ref{auxfunc-eMD}),
the DA model (\ref{auxfunc-DA}), or a mixture of them.
Based on these results, we will propose a family of methods in the unifying framework
for the non-smooth problems in Section~\ref{sec-nonsm} and for
the structured problems in Section~\ref{sec-struc}.

The arguments of this section can be applied to both non-smooth and structured problems.

For a point $y \in Q$, denote by $l_f(y;\cdot): E \rightarrow \Real\cup\{+\infty\}$ a proper lower semicontinuous convex function with $f(x) \geq l_f(y;x),~ \forall x \in Q$, \ie, a lower convex approximation of $f(x)$ on $Q$ at $y \in Q$. We do not require any other assumption on $l_f(y;x)$ in this section but, in Sections~\ref{sec-nonsm} and~\ref{sec-struc}, we further require the assumptions (\ref{l_f-nonsm}) for the non-smooth problems and (\ref{l_f-struc}) for the structured problems, respectively (see also Example \ref{ex-struc} for more specific forms of $l_f(y;x)$).

For the prox-function $d(x)$, we denote $l_d(y;x) := d(y)+\fv{\nabla{d}(y)}{x-y}$.
Note that $d(x) \geq l_d(y;x)$ and $\xi(y,x) = d(x) - l_d(y;x)$ for any $x,y \in Q$.

We introduce the following two kinds of ``parameters'' which will be used in our methods.
Later on, they will be tuned to obtain an appropriate convergence rate for the methods.
\begin{itemize}
\item[-] The {\it weight parameter} $\{\lambda_k\}_{k \geq 0}$. We assume that $\lambda_k > 0$ for all $k \geq 0$
\item[-] The {\it scaling parameter} $\{\beta_k\}_{k \geq -1}$. We assume that $\beta_{k} \geq \beta_{k-1} > 0$ for all $k \geq 0$.
\end{itemize}
Note that the sequence of scaling parameters $\{\beta_k\}_{k \geq -1}$ is assumed to be non-decreasing throughout this paper.
For the weight parameter $\{\lambda_k\}_{k \geq 0}$, we define $S_k := \sum_{i=0}^k \lambda_i$.

We use $\{\hat{x}_k\}_{k \geq 0}\subset Q$ and $\{x_k\}_{k \geq 0} \subset Q$ for sequences of {\it approximate solutions} and {\it test points} (for which we compute the (sub)gradients), respectively. (Recall that $x_0 := \argmin_{x \in Q}d(x)$.)

Finally, we consider auxiliary functions $\psi_k(x)$ whose unique minimizers on $Q$ are denoted by $z_k := \argmin_{x \in Q}\psi_k(x)$.
The function $\psi_k(x)$ is assumed to be determined by $\{\lambda_i\}_{i=0}^k,~\{\beta_i\}_{i=-1}^k ,~\{x_i\}_{i=0}^k$, and $\{z_i\}_{i=0}^{k-1}$ for each $k \geq 0$.
We also consider $\psi_{-1}(x)$ (and $z_{-1}:=\argmin_{x \in Q}\psi_{-1}(x)$) for convenience.

The following property will be the fundamental one for the construction of auxiliary functions $\{\psi_k(x)\}_{k \geq -1}$ in our unifying framework.

%% property : framework

%\newpage
\begin{property}
\label{framework}
Let $\{\lambda_k\}_{k \geq 0}$ be a sequence of weight parameters, $\{\beta_k\}_{k \geq -1}$ be a sequence of scaling parameters, $\{x_k\}_{k \geq 0}$ be a sequence of test points, and $l_f(y;x)$ be a lower convex approximation of $f(x)$.
Let $\psi_k(x)$ be auxiliary functions which are determined by $\{\lambda_i\}_{i=0}^k,~\{\beta_i\}_{i=-1}^k ,~\{x_i\}_{i=0}^k$, and $\{z_i\}_{i=0}^{k-1}$ where 
$z_i:=\argmin_{x \in Q}\psi_i(x)$ for each $k \geq -1$.
Then the following conditions hold:
\begin{itemize}
\item[(i)] $\min_{x \in Q} \psi_{-1}(x) = 0$ and $z_{-1} = x_0$.
\item[(ii)] The following inequality holds for every $k \geq -1$ :

\begin{equation*}
\forall x \in Q,~\psi_{k+1}(x) \geq \min_{z \in Q}\psi_k(z) + \lambda_{k+1} l_f(x_{k+1};x) + \beta_{k+1} d(x) - \beta_{k}l_d(z_k;x).
\end{equation*}
\item[(iii)] The following inequality holds for every $k \geq 0$ :
\begin{equation}\label{cond3}
\min_{x \in Q}\psi_k(x) \leq \min_{x \in Q}\left\{ \sum_{i=0}^k \lambda_i l_f(x_i;x) + \beta_k l_d(z_k;x) \right\}.
\end{equation}
\end{itemize}
\end{property}

Now, let us see that the auxiliary functions $\{\psi_k(x)\}$ of existing methods shown in Sections~\ref{ssec-opt-algs-nonsm} and~\ref{ssec-opt-algs-sm} can be unified via Property~\ref{framework}.
We propose the following concrete construction of auxiliary functions which will satisfy Property~\ref{framework}:

\begin{equation}\label{auxfunc}
\left.\begin{array}{l}
\textrm{(0) Define } \psi_{-1}(x) := \beta_{-1}d(x).\\
\textrm{(1) For each } k \geq -1, \textrm{ define } \psi_{k+1}(x) \textrm{ by either the {\it extended Mirror-Descent (MD)}} \\ \quad~~\textrm{{\it model} (\ref{auxfunc-eMD}) or the {\it Dual-Averaging (DA) model} (\ref{auxfunc-DA}).}
\end{array}\right\}
\end{equation}
{\bf Extended MD model:}
\begin{equation}\label{auxfunc-eMD}
\psi_{k+1}(x) := \min_{z \in Q}\psi_k(z) + \lambda_{k+1}l_f(x_{k+1};x) + \beta_{k+1} d(x) - \beta_{k}l_d(z_k;x).
\end{equation}
{\bf DA model:}
\begin{equation}\label{auxfunc-DA}
\psi_{k+1}(x) := \psi_k(x) + \lambda_{k+1}l_f(x_{k+1};x) + \beta_{k+1}d(x) - \beta_{k}d(x).
\end{equation}

In both cases, $\psi_{k+1}(x)$ is a proper lower semicontinuous and strongly convex function on $Q$.

%% remark

\begin{remark}\label{remark-auxfunc}
The construction (\ref{auxfunc}) includes the following particular ones.
\begin{itemize}

\item %
Constructing $\{\psi_k(x)\}$ by (\ref{auxfunc}) with only the extended MD model updates (\ref{auxfunc-eMD}) yields
\begin{equation}
\label{eMD-model}
\begin{array}{rcl}
\psi_k(x) &=& \min_{z \in Q}\psi_{k-1}(z) + \lambda_k l_f(x_k;x) + \beta_k d(x) - \beta_{k-1}l_d(z_{k-1};x),\\
z_k &=& \argmin_{x \in Q}\big\{ \lambda_{k} l_f(x_{k};x) + \beta_k d(x) - \beta_{k-1}l_d(z_{k-1};x) \big\}.
\end{array}
\end{equation}
Because $\xi(z_{k-1},x)=d(x)-l_d(z_{k-1};x)$, the definition of $\psi_k(x)$ coincides with, up to a constant addition, the one of
the MDM (\ref{MD-org-auxfunc}) when
$\beta_k = 1$ and $x_k = z_{k-1}$, and
with the one of the Tseng's second APG method (\ref{MD-auxfunc-sm}) for
$\beta_k = L/\sigma$.

\item %
Constructing $\{\psi_k(x)\}$ by (\ref{auxfunc}) with only the DA model updates (\ref{auxfunc-DA}) yields
\begin{equation}
\label{DA-model}
\begin{array}{rcl}
\psi_k(x) &=& \sum_{i=0}^k \lambda_i l_f(x_i;x) + \beta_k d(x),\\
z_k &=& \argmin_{x \in Q}\left\{ \sum_{i=0}^k \lambda_i l_f(x_i;x) + \beta_k d(x) \right\}
\end{array}
\end{equation}
which coincides with the one of the DAM (\ref{DA-org-auxfunc}) and with the one of the Tseng's third APG method (\ref{DA-auxfunc-sm}) with $\beta_k = L/\sigma$.
\end{itemize}
Notice that a pure extended MD model updates (\ref{eMD-model}) considers only the previous $l_f(x_k;x)$ while the DA model updates (\ref{DA-model}) accumulates all $l_f(x_i;x)$'s.
Moreover, we can mix the updates (\ref{eMD-model}) and (\ref{DA-model}) in any order which corresponds in selecting {\it some} of previous $l_f(x_i;x)$'s to define the subproblem.

Note that, for a fixed $\psi_k(x)$, the construction (\ref{auxfunc-eMD}) of $\psi_{k+1}(x)$ is the minimalist choice which satisfies Property~\ref{framework}; according to (ii), any auxiliary function $\psi_{k+1}(x)$ majorizes the one defined by (\ref{auxfunc-eMD}) on the set $Q$.
\end{remark}

To prove Property~\ref{framework} for the construction (\ref{auxfunc}) of auxiliary functions, the following lemma \cite[Property~2]{Tseng08} is useful.

%% lemma

\begin{lemma}
\label{prox-ope-lemma}
Let $h: E \rightarrow \Real\cup\{+\infty\}$ be a proper lower semicontinuous convex function with $Q \subset \dom{h}$ and $\beta$ be a positive number. Denote $\psi(x) = h(x) + \beta d(x)$. Then the minimization problem $\min_{x \in Q}\psi(x)$ has a unique solution $z^* \in Q$ and it satisfies
\begin{equation*}
\psi(x) \geq \psi(z^*) + \beta \xi(z^*, x), \quad \forall x \in Q.
\end{equation*}
\end{lemma}

Now we prove the following result which plays a crucial role in the development of our methods.

%% proposition

\begin{proposition} \label{auxfunc-admit}
Any sequence of auxiliary functions $\{\psi_k(x)\}$ constructed by (\ref{auxfunc}) satisfies Property~\ref{framework}.
\end{proposition}

\begin{proof}
% (i)
Since $\min_{x \in Q}d(x) = d(x_0) = 0$, $\psi_{-1}(x) = \beta_{-1}d(x)$ 
satisfies the condition (i) with $z_{-1}=x_0$.

% (ii)
Let us prove the condition (ii) considering two cases for a fixed $k\geq -1$.
If $\psi_{k+1}(x)$ is updated by (\ref{auxfunc-eMD}), then the condition (ii) is satisfied with equality.

Next, consider the case of the update by (\ref{auxfunc-DA}).
Notice that on the construction (\ref{auxfunc}), we can easily check by induction that the functions $h_k(x) := \psi_k(x) - \beta_kd(x)$ are always proper lower semicontinuous and convex. Thus Lemma~\ref{prox-ope-lemma} implies that
\begin{equation}\label{auxfunc-sconv}
\psi_k(x) \geq \min_{z \in Q}\psi_k(z) + \beta_k \xi(z_k,x),\quad \forall x \in Q.
\end{equation}
Therefore, we obtain
\begin{eqnarray*}
\psi_{k+1}(x) &=& \psi_k(x) + \lambda_{k+1}l_f(x_{k+1};x) + \beta_{k+1} d(x) - \beta_k d(x)\\
&\stackrel{(\ref{auxfunc-sconv})}{\geq}& [\min_{z \in Q}\psi_k(z) + \beta_k\xi(z_k,x)] + \lambda_{k+1}l_f(x_{k+1};x) + \beta_{k+1}d(x) - \beta_kd(x)\\
&=& \min_{z \in Q}\psi_k(z) + \lambda_{k+1}l_f(x_{k+1};x) + \beta_{k+1}d(x) - \beta_kl_d(z_k;x)
\end{eqnarray*}
for all $x \in Q$ and (ii) is satisfied.

% (iii)
Let us finally prove the condition (iii), namely, prove the inequality (\ref{cond3}) for all $k \geq 0$.
We actually show that (\ref{cond3}) is also valid for all $k \geq -1$.
The case $k = -1$ is due to the optimality condition for $z_{-1} = \argmin_{x \in Q}\psi_{-1}(x) = \argmin_{x \in Q}\beta_{-1}d(x)$, that is, $\min_{x \in Q}\beta_{-1}d(x) = \min_{x \in Q}\beta_{-1}l_d(z_{-1};x)$ holds.

Next, we introduce the index set
\[K:=\{k~|~k\in\{-1,0,1,\ldots\},~ \psi_{k}(x)\text{ is updated by (\ref{auxfunc-eMD})}\}\]
and we divide to two cases: $K=\emptyset$ and $K\ne\emptyset$.

The case $K=\emptyset$, which corresponds to the construction (\ref{DA-model}), \ie, $\psi_k(x)=\sum_{i=0}^k \lambda_i l_f(x_i;x)+\beta_k d(x)$, is immediate to prove (iii) as follow:
\[
\min_{z \in Q}\psi_k(z)
\stackrel{(\ref{auxfunc-sconv})}{\leq} \psi_k(x) -\beta_k \xi(z_k,x)
= \sum_{i=0}^k \lambda_i l_f(x_i;x) + \beta_kl_d(z_k;x)
\]
for every $x \in Q$ and $k \geq 0$.

Suppose now that $\emptyset \ne K=\{k_1,k_2,\ldots\}$ where $k_1 < k_2<\ldots$.
When $K$ is finite, say $K=\{k_1,\ldots,k_m\}$, we define $k_{m+1}:=+\infty$.
Then, it suffices to prove the following fact by induction on $i=1,2,\ldots$~:
\[(P_i): \text{ the inequality (\ref{cond3}) holds for all } k \text{ with } -1 \leq k < k_i.\]

The proof of $(P_1)$ is as follows.
If $k_1=0$, then $(P_1)$ corresponds to the inequality (\ref{cond3}) for $k=-1$, which was just proved.
If $k_1>0$, then $\psi_0(x),\ldots,\psi_{k_1-1}(x)$ are constructed by only the DA model (\ref{auxfunc-DA}) from which $(P_1)$ follows as the same way as the case $K=\emptyset$. 

Assume that $(P_i)$ is true for some $i \geq 1$. By the definition of the set $K$, we know that $\psi_{k}(x)$ is constructed by (\ref{auxfunc-eMD}) for $k=k_i$ and by (\ref{auxfunc-DA}) for $k$ with $k_{i} < k < k_{i+1}$ (Recall that, when $K$ is finite and $i=|K|$, we have $k_{i+1}=+\infty$).
Therefore, $\psi_{k}(x)$ for $k_i \leq k < k_{i+1}$ is defined by
\begin{eqnarray*}
\psi_k(x) &=&
	\left[
	\min_{z \in Q}\psi_{k_i-1}(z) + \lambda_{k_i}l_f(x_{k_i};x)+\beta_{k_i}d(x)-\beta_{k_i-1}l_d(z_{k_i-1};x)
	\right]\\
&&
	+\sum_{j\,:\,k_i \leq j<k}[\lambda_{j+1}l_f(x_{j+1};x)+\beta_{j+1}d(x)-\beta_j d(x)]\\
&=&
	\min_{z \in Q}\psi_{k_i-1}(z) + \sum_{j=k_i}^k\lambda_jl_f(x_j;x)+\beta_{k}d(x)-\beta_{k_i-1}l_d(z_{k_i-1};x).
\end{eqnarray*}
Note that a summation over the empty set is defined to be zero.
Since (\ref{cond3}) holds for \mbox{$k=k_i-1\geq-1$} by the hypothesis $(P_i)$, we have the following for all $x \in Q$ and $k$ \mbox{with $k_i\leq k < k_{i+1}$:}
\begin{eqnarray*}
\min_{z \in Q}\psi_{k}(z)
&\stackrel{(\ref{auxfunc-sconv})}{\leq}& \psi_{k}(x) - \beta_{k}\xi(z_{k},x)\\
&=& \left[\min_{z \in Q}\psi_{k_i-1}(z) + \sum_{j=k_i}^k\lambda_j l_f(x_j;x) + \beta_{k}d(x) - \beta_{k_i-1}l_d(z_{k_i-1};x)\right] \\
&& - \beta_{k}\xi(z_{k},x) \\
&\stackrel{\rm (\ref{cond3})}{\leq}& \left[\sum_{i=0}^{k_i-1} \lambda_i l_f(x_i;x)+\beta_{k_i-1} l_d(z_{k_i-1};x)\right]\\ && \qquad + \sum_{j=k_i}^k\lambda_i l_f(x_i;x) + \beta_{k}d(x) - \beta_{k_i-1}l_d(z_{k_i-1};x)- \beta_{k}\xi(z_{k},x) \\
&=& \sum_{i=0}^{k} \lambda_i l_f(x_i;x)+\beta_{k} l_d(z_{k};x).
\end{eqnarray*}
Hence, the inequality (\ref{cond3}) holds for $k_i \leq k < k_{i+1}$ which shows $(P_{i+1})$ and therefore (iii).
\end{proof}

%%% Nesterov's concept

To conclude this section, we define the following relation based on the Nesterov's approach \cite{Nes05s}, see also \cite{NS15}. We propose (sub)gradient-based methods which generates approximate solutions $\{\hat{x}_k\} \subset Q$ satisfying the following relation for every $k \geq 0$:
\begin{equation}
\label{est-rel}
(R_k) \quad S_k f(\hat{x}_k) \leq \min_{x \in Q} \psi_k(x) + C_k
\end{equation}
where $C_k$ is defined according to the problem structure.

This relation yields the following lemma which provides a convergence rate for all methods.

%% lemma

\begin{lemma}
\label{gen-bound}
Let $\{\psi_k(x)\}$ be a sequence of auxiliary functions satisfying Property~\ref{framework} associated with weight parameters $\{\lambda_k\}_{k \geq 0}$, scaling parameters $\{\beta_k\}_{k \geq -1}$, test points $\{x_k\}_{k \geq 0}$, and a lower convex approximation $l_f(y;x)$ of $f(x)$.
If a sequence $\{\hat{x}_k\} \subset Q$ satisfies the relation $(R_k)$ for some $k \geq 0$, then we have
\begin{equation*}
f(\hat{x}_k) - f(x^*) \leq \frac{\beta_k l_d(z_k;x^*) + C_k}{S_k}
\end{equation*}
where $z_k:=\argmin_{x \in Q}\psi_k(x)$.
\end{lemma}

\begin{proof}
Since $\sum_{i=0}^k \lambda_il_f(x_i;x) \leq S_k f(x)$ for all $x \in Q$, 
using the condition (iii) of Property~\ref{framework} yields
\[ \min_{x \in Q}\psi_k(x) \leq \min_{x \in Q} \left\{ \sum_{i=0}^k \lambda_i l_f(x_i;x) + \beta_k l_d(z_k;x) \right\} \leq \min_{x \in Q} \left\{ S_kf(x) + \beta_k l_d(z_k;x) \right\} \leq S_kf(x^*) + \beta_k l_d(z_k;x^*). \]
Therefore, the relation $(R_k)$ implies
\begin{equation*}
S_k f(\hat{x}_k) \leq \min_{x \in Q}\psi_k(x) + C_k \leq S_k f(x^*) + \beta_kl_d(z_k;x^*)+C_k.
\end{equation*}
\end{proof}

All the proposed methods are constructed so that they satisfy the relation $(R_k)$ for some $C_k$ (then we can obtain a convergence result from Lemma~\ref{gen-bound}) when the auxiliary function also satisfies Property~\ref{framework}.
Although the relation $(R_k)$ and its variants take a key role in the proofs on the convergence estimates for these methods, we are not certain if they are essential. That is, it might be possible to prove the results without using the relation $(R_k)$. The role of the $(R_k$) as a proving technique is more apparent in the paper \cite{Ito15}.

%%%
% section 4: optimal methods for non-smooth problems
%%%

\section{A family of subgradient-based methods in the unifying framework}
\label{sec-nonsm}

We now focus, in this section, on the non-smooth problems introduced in Section~\ref{ssec-setting} and we establish a novel family of subgradient-based methods for them.
In Section~\ref{ssec-alg-nonsm}, we propose the general subgradient-based method (Methods~\ref{gen-alg-nonsm} (a) and (b)). Then, we analyze the convergence of them in Section~\ref{ssec-conv-nonsm} and finally we do a comparison with existing methods in Section~\ref{ssec-particularize-nonsm}.

Throughout this section, we assume that the problem (\ref{primal-problem}) belongs to the class of non-smooth problems, namely, we have available a subgradient mapping $g(y) \in \partial{f}(y)$ for $y \in Q$
and the function $l_f(y;x)$ defined by (\ref{l_f-nonsm}), and that the subproblem (\ref{subprob-nonsm}) is efficiently solvable (see Section~\ref{ssec-setting}).
Remark that for the non-smooth problems, the subproblems $z_k = \argmin_{x \in Q}\psi_k(x)$ constructed from (\ref{auxfunc}) are of the form (\ref{subprob-nonsm}).

%%% general methods

\subsection{The general subgradient-based methods in the unifying framework}
\label{ssec-alg-nonsm}

Here, we develop update formulas for the test points $\{x_k\}$ and the approximate solutions $\{\hat{x}_k\}$ so that they satisfy the relation $(R_k)$ with
\begin{equation}
\label{C_k-nonsm}
C_k = \frac{1}{2 \sigma}\sum_{i=0}^k \frac{\lambda_i^2}{\beta_{i-1}}\|g_i\|_*^2,
\end{equation}
where $g_k:=g(x_k) \in \partial{f}(x_k)$.
We also use the following alternative relation:
\begin{equation}\label{hat-R-nonsm}
(\hat{R}_k) ~~ \sum_{i=0}^k \lambda_i f(x_i) \leq \min_{x \in Q}\psi_k(x) + C_k.
\end{equation}
Note that the relation $(\hat{R}_k)$ provides an alternative
to Lemma~\ref{gen-bound} which can be
proven in a similar way: if $\{\psi_k(x)\}$ admits Property~\ref{framework} and the relation $(\hat{R}_k)$ is satisfied for some $k \geq 0$, then we have
\begin{equation}
\label{gen-bound-alt-nonsm}
\frac{1}{S_k}\sum_{i=0}^k \lambda_i f(x_i) - f(x^*) \leq \frac{\beta_k l_d(z_k;x^*) + C_k}{S_k}.
\end{equation}

%%% lemma

We use the following lemma for our analysis.
Recall that $\sigma>0$ is the convexity parameter of the prox-function $d(x)$.

\begin{lemma}
\label{tool-nonsm}
Let $\{x_k\}_{k \geq 0} \subset Q$ and $g_k \in \partial{f}(x_k),~ k \geq 0$.
Then, for $\lambda \in \Real,~ \beta > 0$ and $x, z \in Q$, we have
\begin{equation*}
\fv{\lambda g_k}{x-z} + \beta\xi(z,x) + \frac{1}{2 \sigma \beta}\norm{\lambda g_k}_*^2  \geq 0, \quad\forall k \geq 0,
\end{equation*}
and, in particular,
\begin{equation*}
\lambda l_f(x_k;x) + \beta \xi(x_k,x) + \frac{\lambda^2}{2 \sigma \beta}\norm{g_k}_*^2 \geq \lambda f(x_k), \quad \forall k \geq 0.
\end{equation*}
\end{lemma}

\begin{proof}
Since for every $x \in E$ and $s \in E^*$ the inequality $\frac{1}{2}\|x\|^2 + \frac{1}{2}\|s\|_*^2 \geq \fv{s}{x}$ holds, we have
\begin{equation*}
\fv{\lambda g_k}{x-z} + \beta\xi(z,x) + \frac{1}{2 \sigma \beta}\norm{\lambda g_k}_*^2
\geq
\fv{\lambda g_k}{x-z} + \frac{\sigma\beta}{2}\norm{x-z}^2 + \frac{1}{2 \sigma \beta}\norm{\lambda g_k}_*^2 \geq 0.
\end{equation*}
Substituting $z=x_k$ for this inequality and adding $\lambda f(x_k)$ to  both sides, we obtain the second assertion.
\end{proof}

Now, let us show the following key result which will
provide efficient subgradient-based methods in a 
straightforward way.

%%% theorem

\begin{theorem}
\label{estrel-anal-nonsm}
Let $\{\psi_k(x)\}$ be a sequence of auxiliary functions satisfying Property~\ref{framework} associated with weight parameters $\{\lambda_k\}_{k \geq 0}$, scaling parameters $\{\beta_k\}_{k \geq -1}$, test points $\{x_k\}_{k \geq 0}$, and the lower convex approximation $l_f(y;x)$ of $f(x)$ defined by (\ref{l_f-nonsm}).
Denote $z_k = \argmin_{x \in Q}\psi_k(x)$ and define $C_k$ by (\ref{C_k-nonsm}).
Then the following assertions hold.
\begin{itemize}
\item[(a)] The relations $(R_0)$ and $(\hat{R}_0)$ are satisfied 
by setting $\hat{x}_0 := x_0$.
\item[(b)] Suppose that the relation $(R_k)$ is satisfied for some integer $k \geq 0$.
If the relation $x_{k+1}=z_k$ holds, then the relation $(R_{k+1})$ is satisfied by setting
\[
\hat{x}_{k+1} := \ds\frac{S_k \hat{x}_k + \lambda_{k+1}x_{k+1}}{S_{k+1}}.
\]
Moreover, if the relations $(\hat{R}_k)$ 
is satisfied for some $k\geq 0$
and $x_{k+1} = z_k$ holds, then $(\hat{R}_{k+1})$ is satisfied.
\item[(b')] Suppose that the relation $(R_k)$ is satisfied for some integer $k \geq 0$.
If the relation
\[
x_{k+1} = \frac{S_k \hat{x}_k + \lambda_{k+1}z_{k}}{S_{k+1}}
\]
holds, then the relation $(R_{k+1})$ is satisfied by setting $\hat{x}_{k+1}:=x_{k+1}$.
\end{itemize}
\end{theorem}

\begin{proof}
For a test point $x_k \in Q$, we denote $g_k:=g(x_k) \in \partial{f}(x_k)$.

We remark that using the condition (ii) of Property~\ref{framework} we obtain the inequality
\[ \forall k \geq -1,~~\min_{x \in Q}\psi_{k+1}(x) \geq \min_{x \in Q}\psi_k(x) + \lambda_{k+1}l_f(x_{k+1},z_{k+1}) + \beta_k\xi(z_{k},z_{k+1}) \]
by setting $x=z_{k+1}=\argmin_{x \in Q}\psi_{k+1}(x)$ (recall that $d(x) \geq 0~(x \in Q)$ and $\beta_{k+1}\geq\beta_k$).\\
(a)~ Letting $k = -1$ in the condition (ii) and using the condition (i) of Property~\ref{framework}, we have
\begin{eqnarray*}
\min_{x \in Q}\psi_0(x) + \frac{\lambda_0^2}{2 \sigma \beta_{-1}}\norm{g_0}_*^2
	&\geq& \big[\min_{x \in Q}\psi_{-1}(x) + \lambda_0 l_f(x_0;z_0) + \beta_{-1} \xi(z_{-1},z_0)\big]
		 + \frac{\lambda_0^2}{2 \sigma \beta_{-1}}\norm{g_0}_*^2 \\
	&=& \lambda_0 l_f(x_0;z_0) + \beta_{-1} \xi(z_{-1},z_0)
		 + \frac{\lambda_0^2}{2 \sigma \beta_{-1}}\norm{g_0}_*^2 \\
	&=& \lambda_0 l_f(x_0;z_0) + \beta_{-1} \xi(x_0,z_0)
		 + \frac{\lambda_0^2}{2 \sigma \beta_{-1}}\norm{g_0}_*^2 \\
	&\geq& \lambda_0 f(x_0) \\
	&=& S_0 f(\hat{x}_0),
\end{eqnarray*}
where the last inequality is due to Lemma~\ref{tool-nonsm}. \\
(b)~ By the condition (ii) of Property~\ref{framework} and the assumptions for $x_{k+1}$ and $\hat{x}_{k+1}$, we obtain that
\\[1em]
$\ds\min_{x \in Q}\psi_{k+1}(x) + \frac{1}{2 \sigma}\sum_{i=0}^{k+1}\frac{\lambda_i^2}{\beta_{i-1}}\norm{g_i}_*^2$ \\[-	1em]
\begin{eqnarray}
	&\geq& \left[ \min_{x \in Q}\psi_k(x) + \lambda_{k+1}l_f(x_{k+1};z_{k+1}) + \beta_k \xi(z_{k},z_{k+1}) \right]
	+ \frac{\lambda_{k+1}^2}{2 \sigma \beta_k}\norm{g_{k+1}}_*^2 
	+ \frac{1}{2 \sigma}\sum_{i=0}^k\frac{\lambda_i^2}{\beta_{i-1}}\norm{g_i}_*^2
	\nonumber \\
	&=& \min_{x \in Q}\psi_k(x)
	+ \frac{1}{2 \sigma}\sum_{i=0}^k\frac{\lambda_i^2}{\beta_{i-1}}\norm{g_i}_*^2
	+ \left[  \lambda_{k+1}l_f(x_{k+1};z_{k+1}) + \beta_k \xi(x_{k+1},z_{k+1})  + \frac{\lambda_{k+1}^2}{2 \sigma \beta_k}\norm{g_{k+1}}_*^2 \right] \nonumber\\
	&\geq&
	\left[
	\min_{x \in Q}\psi_k(x)
	+ \frac{1}{2 \sigma}\sum_{i=0}^k\frac{\lambda_i^2}{\beta_{i-1}}\norm{g_i}_*^2
	\right]
	+ \lambda_{k+1}f(x_{k+1})\nonumber\\
	&\geq& S_k f(\hat{x}_k) + \lambda_{k+1} f(x_{k+1})\nonumber\\ % ***
	&\geq& S_{k+1} f\left(\frac{S_k \hat{x}_k + \lambda_{k+1}x_{k+1}}{S_{k+1}}\right)\nonumber\\
	& = & S_{k+1} f(\hat{x}_{k+1}) \nonumber,
\end{eqnarray}
where we used Lemma~\ref{tool-nonsm}, the relation $(R_k)$, and the convexity of $f$ in the last three inequalities, respectively.
This implies that the relation $(R_{k+1})$ holds.
Moreover, replacing the use of $(R_k)$ by $(\hat{R}_k)$ in the above
inequality, we obtain the relation $(\hat{R}_{k+1})$, which proves the latter assertion.\\
(b')~Denote $x'_{k+1}=\ds\frac{S_k \hat{x}_k + \lambda_{k+1} z_{k+1}}{S_{k+1}}$. Then the relation $x_{k+1}=\ds\frac{S_k \hat{x}_k + \lambda_{k+1} z_{k}}{S_{k+1}}$ yields
\[ z_{k+1} - z_{k} = \frac{S_{k+1}}{\lambda_{k+1}}(x'_{k+1}-x_{k+1}). \]
Thus the condition (ii) of Property~\ref{framework} and the relation $(R_k)$ imply that \\[1em]
$\ds\min_{x \in Q}\psi_{k+1}(x) + \frac{1}{2 \sigma}\sum_{i=0}^{k+1}\frac{\lambda_i^2}{\beta_{i-1}}\norm{g_i}_*^2$ \\[-1em]
\begin{eqnarray*}
	&\geq& \min_{x \in Q}\psi_k(x) + \lambda_{k+1}l_f(x_{k+1};z_{k+1}) + \beta_k \xi(z_{k},z_{k+1})
	+ \frac{\lambda_{k+1}^2}{2 \sigma \beta_k}\norm{g_{k+1}}_*^2
	+ \frac{1}{2 \sigma}\sum_{i=0}^k\frac{\lambda_i^2}{\beta_{i-1}}\norm{g_i}_*^2 \\
	&\geq& S_k f(\hat{x}_k) + \lambda_{k+1}l_f(x_{k+1};z_{k+1}) + \beta_k \xi(z_{k},z_{k+1}) + \frac{\lambda_{k+1}^2}{2 \sigma \beta_k}\norm{g_{k+1}}_*^2 \\
		&\geq& S_k l_f(x_{k+1};\hat{x}_k) + \lambda_{k+1}l_f(x_{k+1};z_{k+1}) + \beta_k \xi(z_{k},z_{k+1}) + \frac{\lambda_{k+1}^2}{2 \sigma \beta_k}\norm{g_{k+1}}_*^2 \\
		&=& S_{k+1} l_f\left(x_{k+1};\frac{S_k \hat{x}_k + \lambda_{k+1}z_{k+1}}{S_{k+1}}\right) + \beta_k \xi(z_{k},z_{k+1}) + \frac{\lambda_{k+1}^2}{2 \sigma \beta_k}\norm{g_{k+1}}_*^2 \\
		&=& S_{k+1} l_f(x_{k+1};x'_{k+1}) + \beta_k \xi(z_{k},z_{k+1}) + \frac{\lambda_{k+1}^2}{2 \sigma \beta_k}\norm{g_{k+1}}_*^2 \\
		&=& S_{k+1} f(x_{k+1}) + \fv{g_{k+1}}{S_{k+1}(x'_{k+1}-x_{k+1})} \\
		&& \hspace{3truecm}  + \beta_k \xi(z_{k},z_{k+1}) + \frac{\lambda_{k+1}^2}{2 \sigma \beta_k}\norm{g_{k+1}}_*^2 \\
		&=& S_{k+1} f(x_{k+1}) + \fv{\lambda_{k+1}g_{k+1}}{z_{k+1}-z_k}  + \beta_k \xi(z_{k},z_{k+1}) + \frac{\lambda_{k+1}^2}{2 \sigma \beta_k}\norm{g_{k+1}}_*^2 \\
		&\geq& S_{k+1} f(x_{k+1})\\
		&=& S_{k+1} f(\hat{x}_{k+1})
\end{eqnarray*}
where the last inequality is due to Lemma~\ref{tool-nonsm}.
\end{proof}

%%% algorithm

Now we are ready to propose the following two novel subgradient-based methods (Method~\ref{gen-alg-nonsm} (a) and (b)) for the non-smooth problems.
\begin{method}[General subgradient-based method]
\label{gen-alg-nonsm}
Suppose that the problem (\ref{primal-problem}) belongs to the class of non-smooth problems (Section~\ref{ssec-setting}).
Choose weight parameters $\{\lambda_k\}_{k \geq 0}$ and scaling parameters $\{\beta_k\}_{k \geq -1}$.
Generate sequences $\{(z_{k-1},x_k,g_k,\hat{x}_k)\}_{k \geq 0}$ by
\begin{equation}
\label{iter-nonsm}
(a) \quad x_k := z_{k-1} := \argmin_{x \in Q}\psi_{k-1}(x),~~ \hat{x}_k := \frac{1}{S_k} \sum_{i=0}^k \lambda_i x_i,~~g_k :=g(x_k) \in \partial{f}(x_k),~~ \textrm{for} ~ k\geq 0
\end{equation}
or by
\begin{equation}
\label{iter-nonsm-alt}
(b) \quad z_{k-1} := \argmin_{x \in Q}\psi_{k-1}(x),~~ \hat{x}_k := x_k := \frac{1}{S_k} \sum_{i=0}^k \lambda_i z_{i-1},~~g_k :=g(x_k) \in \partial{f}(x_k),~~ \textrm{for} ~ k\geq 0
\end{equation}
where $\{\psi_k(x)\}_{k \geq -1}$ is defined using the construction (\ref{auxfunc}) with the lower convex approximation $l_f(y;x)$ of $f(x)$ defined by (\ref{l_f-nonsm}), as well as any construction which admits Property~\ref{framework}.
\end{method}

Notice that the sequences $\{z_k\}_{k\geq -1}$ and $\{x_k\}_{k\geq 0}$ are dummy ones for the methods (a) and (b), respectively, 
but we kept them to preserve the notation.

%%% convergence analysis

\subsection{Convergence analysis of the general subgradient-based method}
\label{ssec-conv-nonsm}

%%% corollary

\begin{corollary}
\label{alg-est-nonsm}
Given the weight parameter $\{\lambda_k\}_{k \geq 0}$, the
scaling parameter $\{\beta_k\}_{k \geq -1}$, and any sequence
$\{(z_{k-1}, x_k, g_k, \hat{x}_k)\}_{k \geq 0}$ generated by

(a) the first procedure (\ref{iter-nonsm}) in Method~\ref{gen-alg-nonsm}, we have: 
\begin{equation}\label{est-nonsm}
f(\hat{x}_k) - f(x^*) \leq \frac{1}{S_k}\sum_{i=0}^k\lambda_if(x_i) - f(x^*) \leq \frac{\beta_{k}l_d(z_k;x^*) + \ds\frac{1}{2 \sigma}\sum_{i=0}^k\frac{\lambda_i^2}{\beta_{i-1}}\norm{g_i}_*^2}{S_{k}}
\end{equation}
for all $k \geq 0$; or \\ 

(b)
the second procedure (\ref{iter-nonsm-alt}) in Method~\ref{gen-alg-nonsm}, we have:
\begin{equation*}%\label{est-nonsm-alt}
f(\hat{x}_k) - f(x^*) \leq \frac{\beta_{k}l_d(z_k;x^*) + \ds\frac{1}{2 \sigma}\sum_{i=0}^k\frac{\lambda_i^2}{\beta_{i-1}}\norm{g_i}_*^2}{S_{k}}
\end{equation*}
for all $k \geq 0$.
\end{corollary}

\begin{proof}
The first inequality in (\ref{est-nonsm}) is from the convexity of $f(x)$.
Proposition~\ref{auxfunc-admit} and
Theorem~\ref{estrel-anal-nonsm} show that the sequences 
generated by the procedures (\ref{iter-nonsm}) and (\ref{iter-nonsm-alt}) satisfy the 
relation $(R_k)$; futhermore, the former construction 
(\ref{iter-nonsm}) also satisfies $(\hat{R}_k)$.
Thus, Lemma~\ref{gen-bound} and the alternative (\ref{gen-bound-alt-nonsm}) of Lemma~\ref{gen-bound} for $(\hat{R}_k)$ prove the assertion.
\end{proof}

In \cite{Nes09}, Nesterov proposed to use of the
auxiliary sequence (\ref{eq:beta}) to ensure an efficient convergence of the DAM. This
sequence also satisfies the identity
\begin{equation}
\label{hat-beta}
\hat{\beta}_k = \sum_{i=-1}^{k-1} \frac{1}{\hat{\beta}_i}~~(k \geq 0)
\end{equation}
and the inequality
\begin{equation}
\label{bound-hat-beta}
\forall k \geq 0,\quad \sqrt{2k+1} \leq \hat{\beta}_k \leq \frac{1}{1+\sqrt{3}}+\sqrt{2k+1}.
\end{equation}

%%% corollary

\begin{corollary}[see also \cite{Nes09}]
\label{param-nonsm}
Consider the following two choices for the
parameters.
\begin{description}

\item[(Simple Averages)]
Let $\{(z_{k-1}, x_k, g_k, \hat{x}_k)\}_{k \geq 0}$ be generated 
by Method~\ref{gen-alg-nonsm} with parameters 
$\lambda_k := 1$ and $\beta_k := \gamma \hat{\beta}_k$ for some $\gamma > 0$. Then we have
\begin{equation}
\label{optconv-nonsm-1}
\forall k \geq 0,\quad f(\hat{x}_k)-f(x^*) \leq \left( \gamma l_d(z_k;x^*) + \frac{M_k^2}{2 \sigma \gamma} \right)\frac{0.5 + \sqrt{2k+1}}{k+1}
\end{equation}
and
\begin{equation}
\label{seq-bd-nonsm-1}
\forall k \geq -1,\quad z_k, x_{k+1}, \hat{x}_{k+1} \in \left\{x \in Q ~:~ 
\norm{x-x^*}^2 \leq  \frac{2d(x^*)}{\sigma}+\frac{M_k^2}{\sigma^2 \gamma^2} \right\}
\end{equation}
where $M_{-1}=0$ and $M_k=\ds\max_{0 \leq i \leq k}\norm{g_i}_*$ for $k \geq 0$.

\item[(Weighted Averages)]
Let $\{(z_{k-1}, x_k, g_k, \hat{x}_k)\}_{k \geq 0}$ be generated by Method~\ref{gen-alg-nonsm} with parameters $\lambda_k := \ds\frac{1}{\norm{g_k}_*}$ and $\ds\beta_k := \frac{\hat{\beta}_k}{\rho\sqrt{\sigma}}$ for some $\rho > 0$. Then we have
\begin{equation}
\label{optconv-nonsm-2}
\forall k \geq 0,\quad f(\hat{x}_k) - f(x^*) \leq M_k\frac{1}{\sqrt{\sigma}} \left( \frac{l_d(z_k;x^*)}{\rho} + \frac{\rho}{2} \right)\frac{0.5 + \sqrt{2k+1}}{k+1}
\end{equation}
and
\begin{equation}
\label{seq-bd-nonsm-2}
\forall k \geq -1,\quad z_k, x_{k+1}, \hat{x}_{k+1}\in \left\{x \in Q ~:~ \norm{x-x^*}^2 \leq \frac{2d(x^*)+\rho^2}{\sigma} \right\}.
\end{equation}
\end{description}

Moreover, for both simple and weighted averages, the above $f(\hat{x}_k)-f(x^*)$'s can be replaced by its upper bound $\frac{1}{S_k}\sum_{i=0}^k\lambda_if(x_i) - f(x^*)$ when we use the first procedure (\ref{iter-nonsm}) in Method~\ref{gen-alg-nonsm}.
In this case, the left hand side of the inequality
can be replaced by $\min\{f(\hat{x}_k)-f(x^*), \min_{0 \leq i \leq k}f(x_i)-f(x^*)\}$.

\end{corollary}

\begin{proof}
Substituting the specified $\lambda_k$ and $\beta_k$ into the estimations in Corollary~\ref{alg-est-nonsm} and using the properties (\ref{hat-beta}) and (\ref{bound-hat-beta}) of $\hat{\beta}_k$, we obtain (\ref{optconv-nonsm-1}) and (\ref{optconv-nonsm-2}), respectively.
Denote by $B_k$ the ball on the right hand side of (\ref{seq-bd-nonsm-1}) for 
$k \geq -1$.
Then $B_{k} \subset B_{k+1}$ for each $k \geq -1$.
The inequality (\ref{optconv-nonsm-1}) implies that $\gamma l_d(z_k;x^*)+(2\sigma\gamma)^{-1}M_k^2 \geq 0$ for all $k \geq 0$, and using the strong convexity, $d(x^*) \geq l_d(z_k;x^*)+\frac{\sigma}{2}\norm{x^*-z_k}^2$, we can obtain that $z_k \in B_k$ for each $k \geq 0$.
We also have $z_{-1} \in B_{-1}$ since $z_{-1}=x_0=\argmin_{x \in Q}d(x)$, $d(z_{-1})=d(x_0)=0$, and $d(x^*) \geq l_d(z_{-1};x^*)+\frac{\sigma}{2}\norm{z_{-1}-x^*}^2 \geq \frac{\sigma}{2}\norm{z_{-1}-x^*}^2$.
Finally, we conclude that $x_{k+1}, \hat{x}_{k+1} \in B_k$ for all $k \geq -1$ because they are convex combinations of $\{z_i\}_{i=-1}^{k}$.
The proof of (\ref{seq-bd-nonsm-2}) is similar.
\end{proof}

\begin{remark}
\label{Nes-bd-strength}
Notice that in our approach, the bounds in (\ref{optconv-nonsm-1})
and (\ref{optconv-nonsm-2}) are slightly smaller than the
ones in (3.3) and (3.5) in \cite{Nes09}, respectively, since 
$l_d(z_{k};x^*)\leq d(x^*)\leq D$.
However, essentially, Nesterov's original argument also arrives to the same bound when $d(x)$ is continuously differentiable on $Q$ (note that the argument in \cite{Nes09} does not impose differentiability on $d(x)$).
In fact, in \cite{Nes09}, Theorems 2 and 3 rely on the estimate (2.15) which is implied from (2.18). Notice in (2.18) that we have
\[ -V_{\beta_{k+1}}(-s_{k+1})=\min_{x \in Q}\{\fv{s_{k+1}}{x-x_0}+\beta_{k+1}d(x)\}=\min_{x \in Q}\{\fv{s_{k+1}}{x-x_0}+\beta_{k+1}l_d(x_{k+1};x)\} \]
by the optimality of $x_{k+1}=\pi_{\beta_{k+1}}(-s_{k+1})$.
Then adding $\sum_{i=0}^k\lambda_i [f(x_i)+\fv{g_i}{x_0-x_i}]$ and using $s_{k+1}=\sum_{i=0}^k \lambda_i g_i$ in (2.18), it yields
\[ \sum_{i=0}^k\lambda_if(x_i) \leq \min_{x \in Q}\left\{\sum_{i=0}^k\lambda_i[f(x_i)+\fv{g_i}{x-x_i}] + \beta_{k+1}l_d(x_{k+1};x)\right\}+\frac{1}{2\sigma}\sum_{i=0}^k\frac{\lambda_i^2}{\beta_i}\norm{g_i}_*^2 \]
which corresponds to the relation $(\hat{R}_k)$\footnote{Notice that $x_{k+1}$ and $\beta_{k+1}$ in \cite{Nes09} are called $z_{k}$ and $\beta_k$ here, respectively.}. Thus we obtained the same bound as our analysis for the DA model.
\end{remark}

A consequence of 
Corollary~\ref{param-nonsm} is that if $M:=\sup\{\|g\|_* : g \in \partial{f}(x),~ x \in Q\}$ is finite,
Method~\ref{gen-alg-nonsm} generates a sequence
$\{\hat{x}_k\}$ such that $f(\hat{x}_k) \to f(x^*)$ with a rate $O(1/\sqrt{k})$ in the number $k$ of iterations.
In particular, if we know an upper bound $R\geq \sqrt{\frac{1}{\sigma}d(x^*)}$ and for the single averages case additionally the $M$, the estimates (\ref{optconv-nonsm-1}) and (\ref{optconv-nonsm-2}) achieve the optimal complexity $O(M^2R^2/\varepsilon^2)$ to obtain an $\varepsilon$-solution for the non-smooth problems when we choose $\gamma:=\frac{M}{\sqrt{2}\sigma R}$ and $\rho:=\sqrt{2\sigma}R$, respectively.
Also Method~\ref{gen-alg-nonsm} with the parameters suggested in Corollary~\ref{param-nonsm} produces bounded sequences $\{x_k\}$, $\{\hat{x}_k\}$, and $\{z_k\}$ (even if $M=+\infty$ for the weighted averages case).

These features are similar to the DAM. We can obtain the optimal convergence rate if we know an upper bound for $d(x^*)$, but without assuming the compactness of $Q$ and fixing the required number of iterations.

%%% The methods of DA and extended MD

\subsection{Particular cases for the extended MD and the DA models}
\label{ssec-particularize-nonsm}

Restricting Method~\ref{gen-alg-nonsm} (a) only to the extended MD 
model (\ref{auxfunc-eMD}), we can obtain the following extension of the MDM.

%% algorithm

\begin{submethod}[Extended Mirror-Descent]
\label{eMD-alg}
Suppose that the problem (\ref{primal-problem}) belongs to the class of non-smooth problems.
Set $x_0 :=\argmin_{x \in Q}d(x)$.
Choose weight parameters $\{\lambda_k\}_{k \geq 0}$ and scaling parameters $\{\beta_k\}_{k \geq -1}$.
Generate sequences $\{(x_k, g_k, \hat{x}_k)\}_{k \geq 0}$ by
\begin{eqnarray*}
g_k &:=& g(x_k) \in \partial{f}(x_k), \\
x_{k+1} &:=& \argmin_{x \in Q}\big\{ \lambda_{k}[f(x_k) + \langle g_k,x-x_k \rangle] + \beta_k d(x) - \beta_{k-1}l_d(x_k;x) \big\},\\
\hat{x}_k &:=& \frac{1}{S_k}\sum_{i=0}^k \lambda_i x_i
\end{eqnarray*}
for $k \geq 0$.
\end{submethod}

The iteration updates described by (\ref{MD-org-iter}) of the original MDM corresponds to the extended MDM (Method~\ref{eMD-alg}) with $\beta_k := 1$.

It is important to note that, according to Corollary~\ref{param-nonsm}, the extended MDM ensures the rate $O(1/\sqrt{k})$ of convergence without fixing a priori the total number of iterations and knowing an upper bound of $d(x^*)$ required for the weight parameters (\ref{MD-org-weights}) of the original MDM.
Furthermore, this advantage holds {\it even if} the feasible region $Q$ is unbounded. The existing averaging techniques \cite{NL14,NJLS09} of the MDM assume the compactness of $Q$ to achieve the same complexity.

If we restrict Method~\ref{gen-alg-nonsm} (a) only to the DA model~(\ref{auxfunc-DA}), we obtain the Nesterov's DAM~(\ref{DA-org-iter}) described in Section~\ref{sssec-DAM}.
In particular, Corollary~\ref{alg-est-nonsm} and 
subsequently Corollary~\ref{param-nonsm} provide a small improvement 
over the original result assuming the differentiability of $d(x)$ (see Remark~\ref{Nes-bd-strength}).
Since our analysis does not introduce the dual space,
the arguments are more straightforward than the original one.

We can also obtain variants of the extended MDM and the DAM from Method~\ref{gen-alg-nonsm} (b).
An interesting feature of these variants is that their convergence results are in relation to the test points $x_k(=\hat{x}_k)$ compared to the average of test points $\hat{x}_k$ of the extended MDM or the DAM. In particular, the variant of the DAM, \ie, Method~\ref{gen-alg-nonsm} (b) with only the DA model (\ref{auxfunc-DA}) corresponds to the double averaging method (\ref{double-ave}) proposed by Nesterov and Shikhman \cite{NS15}.

%%%
% section 5: optimal methods for structured problems
%%%

\section{A family of (inexact) gradient-based methods for structured problems in the unifying framework}
\label{sec-struc}

We now focus, in this section, on the structured problems introduced in Section~\ref{ssec-setting} and  we establish a new family of (inexact) gradient-based methods for them. In Section~\ref{ssec-alg-struc}, we propose the classical and the fast gradient methods (Methods~\ref{gen-alg-clas} and~\ref{gen-alg-fast}, respectively). Then, we analyze the convergence of the proposed methods in Section~\ref{ssec-conv-struc} and finally we do a comparison with existing methods in Section~\ref{ssec-particularize-struc}.

Throughout this section, we suppose that the problem (\ref{primal-problem}) belongs to the class of structured problems, namely, we assume that the inequality (\ref{l_f-struc}) holds for a proper lower semicontinuous convex function $l_f(y;x)$ (a lower convex approximation of $f(x)$) and functions $L(y)>0, \delta(y)\geq 0$ for all $y \in Q$ and we further assume that the subproblem (\ref{gen-subprob}) is efficiently solvable (see Section~\ref{ssec-setting}).

%%% general methods

\subsection{The classical gradient method and the fast gradient method in the unifying framework}
\label{ssec-alg-struc}

Here, we develop update formulas for the test points $\{x_k\}$ and the approximate solutions $\{\hat{x}_k\}$ which will satisfy the relation $(R_k)$ under Property~\ref{framework}.

In this section, we also consider the following alternative of the relation $(R_k)$ for some constant $C_k$:
\begin{equation}\label{hat-R-smooth}
(\hat{R}'_k)~~\sum_{i=0}^k \lambda_if(x_{i+1}) \leq \min_{x \in Q}\psi_k(x) + C_k.
\end{equation}
Notice that the relation $(\hat{R}_k')$ is slightly different from the one $(\hat{R}_k)$ of the non-smooth problems. This relation satisfies the following alternative of Lemma~\ref{gen-bound} (see also (\ref{gen-bound-alt-nonsm})): if $\{\psi_k(x)\}$ satisfies Property~\ref{framework} and the relation $(\hat{R}'_k)$ is satisfied for some $k \geq 0$, then we have
\begin{equation}
\label{gen-bound-alt-smooth}
\frac{1}{S_k}\sum_{i=0}^k \lambda_i f(x_{i+1}) - f(x^*) \leq \frac{\beta_k l_d(z_k,x^*) + C_k}{S_k}.
\end{equation}

The following theorem validates our methods.

%%% theorem

\begin{theorem}
\label{estrel-anal-smooth}
Suppose that the problem (\ref{primal-problem}) belongs to the class of structured problems for a lower convex approximation $l_f(y;x)$ of $f(x)$ and functions $L(y) >0$, $\delta(y)\geq 0$.
Let $\{\psi_k(x)\}_{k \geq -1}$ be a sequence of auxiliary functions satisfying Property~\ref{framework} associated with weight parameters $\{\lambda_k\}_{k \geq 0}$, scaling parameters $\{\beta_k\}_{k \geq -1}$, and test points $\{x_k\}_{k \geq 0}$.
Denote $z_k = \argmin_{x \in Q}\psi_k(x)$.
Then the following assertions hold.
\begin{itemize}

% (a)
\item[(a)] If $\sigma\beta_{-1}/\lambda_0 \geq L(x_0)$ holds, then relation $(R_0)$ is satisfied with $\hat{x}_0:=z_0$ and $C_0:=\lambda_0 \delta(x_0)$ .

% (b')
\item[(b)] Suppose that the relation $(R_k)$ is satisfied for some integer $k \geq 0$.
If the relations $x_{k+1}=z_k$ and $\sigma \beta_k /\lambda_{k+1} \geq L(x_{k+1})$ hold, then the relation $(R_{k+1})$ is satisfied with 
\begin{equation}
\label{update-clas}
\hat{x}_{k+1} := \frac{S_k \hat{x}_k + \lambda_{k+1}z_{k+1}}{S_{k+1}}, \quad C_{k+1} := C_k + \lambda_{k+1}\delta(x_{k+1}).
\end{equation}
\\
Moreover, the same conclusion is valid for $(\hat{R}_k')$ and $(\hat{R}'_{k+1})$ without requiring the relation (\ref{update-clas}) for $\hat{x}_{k+1}$.

% (b)
\item[(b')] Suppose that the relation $(R_k)$ is satisfied for some integer $k \geq 0$.
If the relations
\begin{equation*}
x_{k+1} = \frac{S_k \hat{x}_k + \lambda_{k+1}z_k}{S_{k+1}} \quad
\textrm{and} \quad \sigma \beta_k S_{k+1}/\lambda_{k+1}^2 \geq L(x_{k+1})
\end{equation*}
hold, then the relation $(R_{k+1})$ is satisfied with
\begin{equation*}
\hat{x}_{k+1} := \frac{S_k \hat{x}_k + \lambda_{k+1}z_{k+1}}{S_{k+1}},\quad C_{k+1} := C_k + S_{k+1}\delta(x_{k+1}).
\end{equation*}

\end{itemize}
\end{theorem}

\begin{proof}
Denote $L_k = L(x_k)$ and $\delta_k = \delta(x_k)$.\\
(a)~ The conditions (i) and (ii) of Property~\ref{framework} with $k=-1$ yield that
\begin{eqnarray*}
\min_{x \in Q} \psi_0(x) + \lambda_0 \delta_0
	&\geq& \min_{x \in Q} \psi_{-1}(x) + \lambda_0 l_f(x_0;z_0)  + \beta_{-1} \xi(z_{-1},z_0) + \lambda_0 \delta_0 \\
	&=& \lambda_0\left(l_f(x_0;z_0) + \frac{\beta_{-1}}{\lambda_0}\xi(x_0,z_0) + \delta_0 \right) \\
	&\geq& \lambda_0\left(l_f(x_0;z_0) + \frac{\sigma\beta_{-1}}{\lambda_0}\frac{1}{2}\norm{z_0 - x_0}^2 + \delta_0 \right) \\
	&\geq& \lambda_0\left(l_f(x_0;z_0) + \frac{L_0}{2}\norm{z_0 - x_0}^2 + \delta_0 \right) \\
	&\geq& \lambda_0 f(z_0)
	= S_0 f(\hat{x}_0)\\
\end{eqnarray*}
where the last inequality is due to (\ref{l_f-struc}).\\
(b)~ The condition (ii) of Property~\ref{framework} implies that
\begin{eqnarray}
\min_{x \in Q} \psi_{k+1}(x) + C_{k+1}
	&\geq& \min_{x \in Q} \psi_{k}(x) + C_k + \lambda_{k+1} l_f(x_{k+1};z_{k+1}) + \beta_k \xi(z_k,z_{k+1}) + \lambda_{k+1}\delta_{k+1} \nonumber\\%
	&=& \min_{x \in Q} \psi_k(x) + C_k + \lambda_{k+1} l_f(x_{k+1};z_{k+1}) + \beta_k \xi(x_{k+1},z_{k+1}) + \lambda_{k+1}\delta_{k+1} \nonumber\\%
	&\geq&  \min_{x \in Q} \psi_k(x) + C_k + \lambda_{k+1} \left(l_f(x_{k+1};z_{k+1}) + \frac{\sigma\beta_k}{2\lambda_{k+1}}\norm{z_{k+1} - x_{k+1}}^2 + \delta_{k+1} \right) \nonumber\\%
	&\geq&  \min_{x \in Q} \psi_k(x) + C_k + \lambda_{k+1} \left(l_f(x_{k+1};z_{k+1}) + \frac{L_{k+1}}{2}\norm{z_{k+1} - x_{k+1}}^2 + \delta_{k+1} \right) \nonumber\\%
	&\geq& \min_{x \in Q} \psi_k(x) + C_k + \lambda_{k+1}f(z_{k+1}) \label{ineq-lip}\\%
	&\geq& S_kf(\hat{x}_k) + \lambda_{k+1} f(z_{k+1}) \quad \label{ineq-estrel} \\%
	&\geq& S_{k+1}f\left(\frac{S_k \hat{x}_k + \lambda_{k+1}z_{k+1}}{S_{k+1}}\right) = S_{k+1} f(\hat{x}_{k+1}) \nonumber
\end{eqnarray}
where the inequalities (\ref{ineq-lip}) and (\ref{ineq-estrel}) are due to (\ref{l_f-struc}) and $(R_k)$, respectively.
When we use $(\hat{R}'_k)$ at (\ref{ineq-estrel}), it yields the relation $(\hat{R}'_{k+1})$.\\
(b')~ The assumptions for $x_{k+1}$ and $\hat{x}_{k+1}$ imply $z_{k+1}-z_k = \frac{S_{k+1}}{\lambda_{k+1}}(\hat{x}_{k+1}-x_{k+1})$. Thus, from the condition (ii) of Property~\ref{framework} and the relation $(R_k)$, we obtain
\begin{eqnarray*}
\min_{x \in Q} \psi_{k+1}(x) + C_{k+1}
	&\geq& \min_{x \in Q} \psi_{k}(x) + C_k + \lambda_{k+1} l_f(x_{k+1};z_{k+1}) + \beta_k \xi(z_k,z_{k+1}) + S_{k+1}\delta_{k+1} \\
	&\geq& S_kf(\hat{x}_k) + \lambda_{k+1} l_f(x_{k+1};z_{k+1}) + \beta_k \xi(z_k,z_{k+1}) + S_{k+1}\delta_{k+1} \\
	&\geq& S_kl_f(x_{k+1};\hat{x}_k) + \lambda_{k+1} l_f(x_{k+1};z_{k+1}) + \beta_k \xi(z_k,z_{k+1}) + S_{k+1}\delta_{k+1} \\
	&\geq& S_{k+1}l_f\left(x_{k+1};\frac{S_k \hat{x}_k + \lambda_{k+1} z_{k+1}}{S_{k+1}}\right) + \beta_k \xi(z_k,z_{k+1}) + S_{k+1}\delta_{k+1} \\
	&\geq& S_{k+1} l_f(x_{k+1};\hat{x}_{k+1}) + \frac{\sigma \beta_k}{2}\norm{z_{k+1} - z_k}^2 + S_{k+1}\delta_{k+1} \\
	&=& S_{k+1} \left( l_f(x_{k+1};\hat{x}_{k+1}) + \frac{\sigma \beta_k S_{k+1}}{2\lambda_{k+1}^2}\norm{\hat{x}_{k+1} - x_{k+1}}^2 + \delta_{k+1} \right) \\
	&\geq& S_{k+1} \left( l_f(x_{k+1};\hat{x}_{k+1}) + \frac{L_{k+1}}{2}\norm{\hat{x}_{k+1} - x_{k+1}}^2 + \delta_{k+1} \right) \\
	&\geq& S_{k+1} f(\hat{x}_{k+1}).
\end{eqnarray*}
\end{proof}

Now we are ready to propose the following two general gradient-based methods.

%%% algorithm

\begin{method}[Classical Gradient Method (CGM)]
\label{gen-alg-clas}
Suppose that the problem (\ref{primal-problem}) belongs to the class of structured problems.
Choose weight parameters $\{\lambda_k\}_{k \geq 0}$ and scaling parameters $\{\beta_k\}_{k \geq -1}$.
Generate sequences $\{(z_{k-1},x_k,\hat{x}_k)\}_{k\geq 0}$
by setting
\begin{equation*}
x_k := z_{k-1} := \argmin_{x \in Q}\psi_{k-1}(x),\quad \hat{x}_k := \frac{1}{S_k}\sum_{i=0}^k \lambda_i x_{i+1},
\end{equation*}
for $k\geq 0$, where $\{\psi_k(x)\}_{k \geq -1}$ is defined using the construction (\ref{auxfunc}) as well as any construction which admits Property~\ref{framework}.

\end{method}

%%% algorithm

\begin{method}[Fast Gradient Method (FGM)]
\label{gen-alg-fast}
Suppose that the problem (\ref{primal-problem}) belongs to the class of structured problems.
Choose weight parameters $\{\lambda_k\}_{k \geq 0}$ and scaling parameters $\{\beta_k\}_{k \geq -1}$. Set $x_0 := z_{-1} := \argmin_{x \in Q} d(x)$ and $\hat{x}_0 := z_0 := \argmin_{x \in Q}\psi_0(x)$.
Generate sequences $\{(z_{k-1},x_k,\hat{x}_k)\}_{k\geq 0}$ by setting
\begin{eqnarray*}
x_{k+1} &:=& \frac{\sum_{i=0}^k \lambda_i z_i + \lambda_{k+1}z_k}{S_{k+1}}, \\
z_{k+1} &:=& \argmin_{x \in Q} \psi_{k+1}(x), \\
\hat{x}_{k+1} &:=& \frac{1}{S_{k+1}}\sum_{i=0}^{k+1}\lambda_i z_i,
\end{eqnarray*}
for $k \geq 0$ where $\{\psi_k(x)\}_{k \geq -1}$ is defined using the construction (\ref{auxfunc}) as well as any construction which admits Property~\ref{framework}.
\end{method}

Notice that the sequence $\{z_k\}_{k\geq -1}$
is a dummy one for the CGM.

%%% Convergence analysis

\subsection{Convergence analysis of the CGM and the FGM}
\label{ssec-conv-struc}

By the same observation as Corollary~\ref{alg-est-nonsm}, combining Theorem~\ref{estrel-anal-smooth} and Lemma~\ref{gen-bound} (or the alternative (\ref{gen-bound-alt-smooth}) of Lemma~\ref{gen-bound}), we arrive at the following estimates.

%%% corollary

\begin{corollary}
\label{alg-est-smooth}
Suppose that the problem (\ref{primal-problem}) belongs to the class of structured problems.

\begin{itemize}

\item[(a)]
Let $\{(z_{k-1},x_k,\hat{x}_k)\}_{k \geq 0}$ be generated by the CGM associated with weight parameters $\{\lambda_k\}_{k \geq 0}$ and scaling parameters $\{\beta_k\}_{k \geq -1}$.
If $\sigma \beta_{k-1} /\lambda_k \geq L(x_k)$ holds for all $k \geq 0$, then we have
\begin{equation*}
\forall k \geq 0, ~~ f(\hat{x}_k) - f(x^*) \leq \frac{1}{S_k}\sum_{i=0}^k\lambda_if(x_{i+1}) - f(x^*) \leq \frac{\beta_k l_d(z_k;x^*) + \sum_{i=0}^k \lambda_i \delta(x_i)}{S_k}.
\end{equation*}

\item[(b)]
Let $\{(z_{k-1},x_k,\hat{x}_k)\}_{k \geq 0}$ be generated by the FGM associated with weight parameters $\{\lambda_k\}_{k \geq 0}$ and scaling parameters $\{\beta_k\}_{k \geq -1}$.
If $\sigma \beta_{k-1} S_k/\lambda_k^2 \geq L(x_k)$ holds for all $k \geq 0$, then we have
\begin{equation*}
\forall k \geq 0, ~~ f(\hat{x}_k) - f(x^*) \leq \frac{\beta_k l_d(z_k;x^*) + \sum_{i=0}^k S_i\delta(x_i)}{S_k}.
\end{equation*}

\end{itemize}
\end{corollary}

%%% efficient choices of parameters
Particular choices for the parameters $\lambda_k$ and $\beta_k$ in the above estimates simplify the situation.

%%% corollary

\begin{corollary}
\label{param-smooth}
Suppose that the problem (\ref{primal-problem}) belongs to the class of structured problems in the special case $L(\cdot) \equiv L >0$ and $\delta(\cdot)\equiv \delta \geq 0$.
\begin{itemize}

\item[(a)] Any sequence $\{(z_{k-1},x_k,\hat{x}_k)\}_{k \geq 0}$ generated by the CGM with $\lambda_k := 1$ and $\beta_k := L/\sigma$ satisfies
\begin{equation}
\label{optconv-clas}
\forall k\geq 0,\quad f(\hat{x}_k) - f(x^*) \leq \frac{1}{k+1}\sum_{i=0}^k f(x_{i+1}) - f(x^*) \leq \frac{Ll_d(z_k;x^*)}{\sigma (k+1)} + \delta
\end{equation}
and
\begin{equation}
\label{seq-bd-clas}
\forall k \geq -1,\quad z_k, x_{k+1}, \hat{x}_{k} \in \left\{x \in Q ~:~ \norm{x-x^*}^2 \leq \frac{2d(x^*)}{\sigma} + \frac{2\delta}{L}(k+1) \right\}.
\end{equation}

\item[(b)] Any sequence $\{(z_{k-1}, x_k, \hat{x}_k)\}_{k \geq 0}$ generated by the FGM with $\lambda_{k} := \frac{k+1}{2}$ and $\beta_k := L/\sigma$ satisfies
\begin{equation}
\label{optconv-fast}
\forall k\geq 0,\quad f(\hat{x}_k) - f(x^*) \leq \frac{4Ll_d(z_k;x^*)}{\sigma (k+1)(k+2)} + \frac{k+3}{3}\delta
\end{equation}
and
\begin{equation}
\label{seq-bd-fast}
\forall k \geq -1,\quad z_k, x_{k+1}, \hat{x}_{k} \in \left\{x \in Q ~:~ \norm{x-x^*}^2 \leq \frac{2d(x^*)}{\sigma} + \frac{\delta}{6L} (k+1)(k+2)(k+3)\right\}.
\end{equation}
\end{itemize}
\end{corollary}

\begin{proof}
The estimations (\ref{optconv-clas}) and (\ref{optconv-fast}) can be obtained by substituting the specified parameters to Corollary~\ref{alg-est-smooth}.
By a similar argument as the proof of Corollary~\ref{param-nonsm}, remarking that $x_k \in {\rm conv}\{z_i\}_{i=-1}^{k-1}$ and $\hat{x}_k \in {\rm conv}\{z_i\}_{i=0}^k$, we have (\ref{seq-bd-clas}) and (\ref{seq-bd-fast}).
\end{proof}

Let us consider the case $\delta = 0$ in Corollary~\ref{param-smooth}. This includes the case of a minimization of a convex function with a Lipschitz continuous gradient.
Then the FGM ensures the optimal convergence rate $f(\hat{x}_k)-f(x^*) \leq O\left(\dfrac{LR^2}{k^2}\right)$ where $R=\sqrt{\frac{1}{\sigma}d(x^*)}$ which is faster than the rate $O\left(\dfrac{LR^2}{k}\right)$ guaranteed by the CGM. 
Corollary~\ref{param-smooth} also ensures that the generated sequences {$\{z_k\}$, $\{x_k\}$, and $\{\hat{x}_k\}$ are bounded when $\delta = 0$.

In the case $\delta>0$, a comparison between the CGM and the FGM is not obvious; an immediate fact is that the upper bound in (\ref{optconv-fast}) diverges while the one in (\ref{optconv-clas}) converges to $\delta$. There is a detailed discussion about different situations in \cite[Section 6]{DGN14}.

%%% DA and MD for structured problem

\subsection{Particular cases for the extended MD and the DA models}
\label{ssec-particularize-struc}

The CGM and the FGM yield some existing methods by adopting particular choices for the auxiliary functions $\{\psi_k(x)\}$.

When we apply the extended MD model (\ref{auxfunc-eMD}) and the DA model (\ref{auxfunc-DA}) to the CGM, it yields the iteration updates
\begin{equation}\label{eMD-iter-smooth}
x_{k+1} := \argmin_{x \in Q}\big\{ \lambda_k l_f(x_k;x) + \beta_k d(x) - \beta_{k-1}l_d(x_k;x) \big\}
\end{equation}
and
\begin{equation*}%\label{DA-iter-smooth}
x_{k+1} := \argmin_{x \in Q}\left\{ \sum_{i=0}^k \lambda_i l_f(x_i;x) + \beta_k d(x) \right\},
\end{equation*}
respectively (recall (\ref{eMD-model}) and (\ref{DA-model})).

In the composite structure (\ref{obj-composite}), these updates with the choice of $\lambda_k$'s and $\beta_k$'s as in Corollary~\ref{param-smooth} (a) yield the {\it primal} and {\it dual gradient methods} analyzed by Nesterov \cite{Nes13} with known Lipschitz constants.
In the Euclidean setting (\ie, $E$ is a Euclidean space, the norm $\norm{\cdot}$ is induced by its inner product, and $d(x)=\frac{1}{2}\norm{x-x_0}^2$), the extended MD update (\ref{eMD-iter-smooth}) is also closely related to the proximal point method proposed by Fukushima and Mine \cite{FM81}.
In fact, assuming the same conditions in \cite[Corollary at p.996]{FM81}, this method is equivalent to the CGM with $\lambda_k:=1/c_k$ and $\beta_k := 1$.}

The above updates in the Euclidean setting also correspond to the {\it primal} and {\it dual gradient methods} \cite{DGN14} for the inexact oracle model (\ref{i-ora}) by choosing $\lambda_k:=1/L(x_k)$ and $\beta_k :=1/\sigma$.
Since $l_d(z_k;x^*) \leq d(x^*)$, Corollary~\ref{alg-est-smooth} for this case provides estimates for the optimal values with smaller upper bounds than those of \cite[Section 4]{DGN14}; for the dual gradient method, in particular, our estimate does not require the computation of the solution ($y_k$ in \cite[Theorem 3]{DGN14}) of another auxiliary subproblem.

The FGM, on the other hand, provides accelerated versions of the above updates derived from the CGM.
Using the extended MD model (\ref{auxfunc-eMD}) for the FGM, it yields the following method\footnote{The variable $z_{-1} (:=x_0)$ disappeared here for simplicity.}.

%%% algorithm MD : accelerated

\begin{submethod}
\label{MD-fast}
Suppose that the problem (\ref{primal-problem}) belongs to the class of structured problems with a lower convex approximation $l_f(y;x)$ of $f(x)$.
Choose weight parameters $\{\lambda_k\}_{k \geq 0}$ and scaling parameters $\{\beta_k\}_{k \geq -1}$. Set $x_0 := \argmin_{x \in Q}d(x)$ and
\(
\hat{x}_0 := z_0 := \argmin_{x \in Q} \big\{ \lambda_0 l_f(x_0;x)
	 + \beta_0 d(x) - \beta_{-1}l_d(x_0;x) \big\}.
\)
Generate sequences $\{(z_{k},x_k,\hat{x}_k)\}_{k \geq 0}$ by setting
\begin{equation*}
\begin{array}{lcl}
x_{k+1} &:=& \ds\frac{\sum_{i=0}^k \lambda_i z_i + \lambda_{k+1}z_k}{S_{k+1}}, \\[4truemm]
z_{k+1} &:=& \argmin_{x \in Q}\big\{ \lambda_{k+1}l_f(x_{k+1};x) + \beta_{k+1} d(x) - \beta_{k}l_d(z_k;x) \big\}, \\[2truemm]
\hat{x}_{k+1} &:=& \ds\frac{1}{S_{k+1}}\sum_{i=0}^{k+1} \lambda_i z_i
\end{array}
\end{equation*}
for $k\geq 0$.
\end{submethod}

The DA model (\ref{auxfunc-DA}), on the other hand, yields the following method.

%%% algorithm DA : accelerated

\begin{submethod}
\label{DA-fast}
Suppose that the problem (\ref{primal-problem}) belongs to the class of structured problems with a lower convex approximation $l_f(y;x)$ of $f(x)$.
Choose weight parameters $\{\lambda_k\}_{k \geq 0}$ and scaling parameters $\{\beta_k\}_{k \geq -1}$.
Set
$x_0 := \argmin_{x \in Q}d(x)$ and
\(
\hat{x}_0 := z_0 := \argmin_{x \in Q} \left\{ \lambda_0 l_f(x_0;x) + \beta_0 d(x) \right\}.
\)
Generate sequences $\{(z_{k},x_k,\hat{x}_k)\}_{k \geq 0}$ by setting
\begin{equation*}
\begin{array}{lcl}
x_{k+1} &:=& \ds\frac{\sum_{i=0}^k \lambda_i z_i + \lambda_{k+1}z_k}{S_{k+1}},\\[4truemm]
z_{k+1} &:=& \argmin_{x \in Q}\ds\left\{ \sum_{i=0}^{k+1} \lambda_il_f(x_i;x) + \beta_{k+1} d(x) \right\}, \\[4truemm]
\hat{x}_{k+1} &:=& \ds\frac{1}{S_{k+1}}\sum_{i=0}^{k+1} \lambda_i z_i
\end{array}
\end{equation*}
for $k\geq 0$.
\end{submethod}

Apparently, Method~\ref{DA-fast} seems to demand a
computation proportional to $k$ to solve the auxiliary subproblem to
obtain each $z_{k+1}$ due to the weighted summation of $l_f(x_i;x)$'s. However,
for all cases considered in Example~\ref{ex-struc}, excepting 
(\ref{l_f-saddle}), the auxiliary subproblems can be simplified to
the form (\ref{gen-subprob}).

When $\{\beta_k\}$ is constant, $L(\cdot)\equiv L,$ and $\delta(\cdot)\equiv 0$, the above two methods are very similar to the Tseng's methods \cite{Tseng08,Tseng10}.
In particular, choosing $\beta_k := L/\sigma$, $\lambda_0:=1$ and $\lambda_{k+1}:=\frac{1+\sqrt{1+4\lambda_k^2}}{2}$ in Methods~\ref{MD-fast} and~\ref{DA-fast}, they yield the Tseng's second and third APG methods (\ref{Tseng-MD}) and (\ref{Tseng-DA}), respectively.
Therefore, we provide a unified way to analyze these methods
while the Tseng's methods require slightly different approaches for
each case.

For the inexact oracle model (\ref{i-ora}), Methods~\ref{MD-fast} and~\ref{DA-fast} can be seen as accelerated versions of primal and dual gradient methods in \cite{DGN14}.
{\it The fast gradient method} in \cite{DGN14} corresponds to a hybrid of these accelerations which requires to solve two subproblems at each iteration. 
Methods~\ref{MD-fast} and~\ref{DA-fast} solve only one subproblem at each iteration preserving the same complexity as the fast gradient method.

%%%
% section 6: concluding remarks
%%%
\section{Concluding remarks}

We have proposed a new family of (sub)gradient-based methods for some classes
of convex optimization problems, which include cases such as non-smooth, smooth, inexact oracle model, composite/saddle structure, {\it etc.}
These methods were bundled under a concept we call unifying framework.

We
also provided a unifying way of analyzing these methods which were performed
separately and independently in the past.
This became possible since we have identified a general relation (Property~\ref{framework}) 
which the auxiliary functions of the mirror-descent and the dual-averaging
methods should satisfy. 
As a by-product, the proposed extended MDM removed the compactness assumption and the fixation of total number of iterations a priori which the variants or the original MDM require to ensure the rate $O(1/\sqrt{k})$ of convergence.

There are infinitely many ways of implementing our methods since
Proposition~\ref{auxfunc-admit} shows that we can freely select from the
extended MD model (\ref{auxfunc-eMD}) or the DA model (\ref{auxfunc-DA}) the $l_f(x_i;x)$'s and the scaled proximal function $d(x)$ to construct each subproblem at each iteration.
All of them achieve the optimal complexity.
Also these methods require a solution of only one subproblem per iteration.

From the viewpoint of the relation (\ref{est-rel}), which we call ($R_k$), the extended mirror-descent model (\ref{auxfunc-eMD}) has a `greedy' feature in the following sense; at each iteration, it attains the smallest upper bound $f(\hat{x}_k) \leq \psi_k(z_k)/S_k$ among those bounds for auxiliary functions satisfying Property~\ref{framework} given the previous one $\psi_{k-1}(x)$.

We list some further consideration to extend our approach as follows:

\begin{itemize}
\item
In order to ensure optimal convergence, our methods require knowing the Lipschitz constant of the gradient of the objective function for the class of structured problems (Section~\ref{sec-struc}). There are, however, some approaches which remove this requirement as observed in \cite{BT09, Nes83, Nes13,Nes15}. One can expect to obtain similar results applying these techniques for the proposed methods.

\item
For the case of convex problems with composite structure considered in 
Beck and Teboulle \cite{BT12}, it is possible to obtain a family of 
{\it smoothing-based first order methods} since our methods correspond to the 
{\it fast iterative method}.

\item
The optimal complexity of (sub)gradient methods depends on the assumptions of the objective function. Development of optimal methods assuming strong convexity of the objective function and/or H\"older continuity of its gradient is one of recent topic of interest \cite{DGNs,DGN14,GL12,GL13,Ito15,JN14,NB,NL14,Nes15}. In particular, a generalization of this paper for such convex problems are discussed in \cite{Ito15}.
\end{itemize}

%%%
% acknowledgements
%%%
\section*{Acknowledgements}
We are thankful to anonymous referees for giving valuable comments and suggestions that clarified the ideas of the paper considerably.
We also thank Nobuo Yamashita for kindly pointing out Paul 
Tseng's earlier work \cite{Tseng08, Tseng10} and
Yurii Nesterov for his work \cite{NS15}.
This work was partially supported by JSPS Grant-in-Aid for 
Scientific Research (C) number 26330024.

%%%
% references
%%%

\end{document}